\documentclass[11 pt]{amsart}
\usepackage{latexsym,amscd,amssymb, graphicx, shuffle, amsthm}  
\usepackage{young}

\usepackage[margin=1in]{geometry}

\numberwithin{equation}{section}

\newtheorem{theorem}{Theorem}[section]
\newtheorem{proposition}[theorem]{Proposition}
\newtheorem{corollary}[theorem]{Corollary}
\newtheorem{lemma}[theorem]{Lemma}
\newtheorem{conjecture}[theorem]{Conjecture}

\newtheorem{problem}[theorem]{Problem}
\newtheorem{example}[theorem]{Example}

\newtheorem{defn}[theorem]{Definition}
\theoremstyle{definition}

\newcommand{\maj}{{\mathrm {maj}}}
\newcommand{\inv}{{\mathrm {inv}}}
\newcommand{\minimaj}{{\mathrm {minimaj}}}

\newcommand{\Des}{{\mathrm {Des}}}
\newcommand{\Val}{{\mathrm {Val}}}
\newcommand{\Rise}{{\mathrm {Rise}}}
\newcommand{\Stir}{{\mathrm {Stir}}}
\newcommand{\Hilb}{{\mathrm {Hilb}}}

\newcommand{\CC}{{\mathbb {C}}}
\newcommand{\ZZ}{{\mathbb {Z}}}

\newcommand{\OP}{{\mathcal{OP}}}

\newcommand{\AAA}{{\mathcal{A}}}

\begin{document}

\title[Ordered set partition statistics and the Delta Conjecture]
{Ordered set partition statistics and the Delta Conjecture}

\author{Brendon Rhoades}
\address
{Department of Mathematics \newline \indent
University of California, San Diego \newline \indent
La Jolla, CA, 92093-0112, USA}
\email{bprhoades@math.ucsd.edu}

\begin{abstract}
The  Delta Conjecture of Haglund, Remmel, and Wilson is a recent generalization of the 
Shuffle Conjecture in the field of diagonal harmonics.  In this paper we give evidence for the Delta 
Conjecture by proving a pair of conjectures of Wilson and Haglund-Remmel-Wilson which give
 equidistribution results for statistics related to inversion count and
major index on objects related to ordered set partitions.  Our results generalize the 
famous result of MacMahon that major index and inversion number share the same distribution on 
permutations.  
\end{abstract}

\maketitle

\section{Introduction}
\label{Introduction}

The recently proven Shuffle Conjecture asserts a formula for the bigraded
Frobenius character of the module of diagonal harmonics in terms of 
a Macdonald eigenoperator and the 
combinatorics of parking functions \cite{CM, HHRLU}.
The Delta Conjecture is a still-open generalization of the Shuffle Conjecture due to Haglund, Remmel,
and Wilson \cite{HRW}.  
This paper proves combinatorial results about statistics on ordered set partitions to obtain a specialization
of the Delta Conjecture. 
We begin by reviewing the Delta Conjecture
and then turn our attention to combinatorics (where it will remain for the rest of the paper).

\subsection{The Delta Conjecture}
Let $\Lambda$ denote the ring of symmetric functions in the variable set $x = (x_1, x_2, \dots )$ with
coefficients in the field 
$\mathbb{Q}(q, t)$.   For any partition $\mu$, let $\tilde{H}_{\mu} = \tilde{H}_{\mu}(x; q, t) \in \Lambda$ 
denote the corresponding 
modified Macdonald symmetric function; it is known that 
\begin{equation*}
\{ \tilde{H}_{\mu} \,:\, \text{$\mu$ a partition}\} 
\end{equation*}
is a basis
for $\Lambda$.

For any symmetric function $f = f(x_1, x_2, \dots) \in \Lambda$,  define the
{\em Delta operator} 
$\Delta'_f: \Lambda \rightarrow \Lambda$ to be the Macdonald eigenoperator given by
\begin{equation}
\Delta'_f: \tilde{H}_{\mu} \mapsto f[B_{\mu}(q,t) - 1] \tilde{H}_{\mu}.  
\end{equation}
Here we are using the plethystic 
shorthand $f[B_{\mu}(q,t) - 1] = f( \dots, q^{i-1} t^{j-1}, \dots )$, where 
$(i, j)$ ranges over the matrix coordinates of all cells  $\neq (1,1)$ in the (English)
Young diagram of $\mu$.

For example, suppose $\mu = (3,2,1) \vdash 6$ is the partition whose Young diagram is shown below. 
We have that
\begin{equation*}
\Delta'_f: \tilde{H}_{(3,2,1)} \mapsto f[B_{(3,2,1)}(q,t) - 1] \tilde{H}_{(3,2,1)} = f(q, q^2, t, qt, t^2) \tilde{H}_{(3,2,1)},
\end{equation*}
where we implicitly set $x_n = 0$ in $f(x_1, x_2, \dots )$ for all $n > 5$.

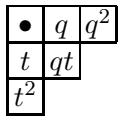
\begin{center}
\begin{Young}
 $\bullet$  & $q$ & $q^2$ \cr
 $t$  & $qt$  \cr 
 $t^2$
\end{Young}
\end{center}

Let $e_d \in \Lambda$ denote the elementary symmetric function of degree $d$.
The Delta Conjecture (\cite[Conjecture 1.1]{HRW}) asserts a combinatorial formula for the 
application of $\Delta'_{e_k}$ to $e_n$ for all nonnegative integers $k < n$.

\begin{conjecture}
{\em (The Delta Conjecture)}
For any integers $n > k \geq 0$,
\begin{align}
\Delta'_{e_k} e_n &= \{z^{n-k-1}\} \left[ \sum_{P \in \mathcal{LD}_n} q^{\mathrm{dinv}(P)} t^{\mathrm{area}(P)}
\prod_{i: a_i(P) > a_{i-1}(P)} \left( 1 + z/t^{a_i(P)} \right) x^P \right] \\
&= \{z^{n-k-1}\} \left[ \sum_{P \in \mathcal{LD}_n} q^{\mathrm{dinv}(P)} t^{\mathrm{area}(P)}
\prod_{i \in \mathrm{Val}(P)} \left( 1 + z/t^{d_i(P) + 1} \right) x^P \right].
\end{align}
Here the operator $\{z^{n-k-1}\}$ extracts the coefficient of $z^{n-k-1}$.
\end{conjecture}

We will not need the set $\mathcal{LD}_n$ of labeled Dyck paths,
nor the statistics $\mathrm{dinv}$, $\mathrm{area}$, $(a_1, \dots, a_n)$, $\mathrm{Val}$, and $(d_1, \dots, d_n)$
that they carry; for details on these objects see \cite{HRW}.
For our purposes it will be enough to know that the Delta Conjecture asserts an equality of 
quasisymmetric
functions in $x = (x_1, x_2, \dots )$ with coefficients depending on the parameters $q$ and $t$.

The two right hand sides of the Delta Conjecture display two  
quasisymmetric functions.  Following \cite{HRW},
we name  these quantities as follows.
\begin{align}
\Rise_{n,k}(x; q,t) &:= \{z^{n-k-1}\} \left[ \sum_{P \in \mathcal{LD}_n} q^{\mathrm{dinv}(P)} t^{\mathrm{area}(P)}
\prod_{i: a_i(P) > a_{i-1}(P)} \left( 1 + z/t^{a_i(P)} \right) x^P \right], \\
\Val_{n,k}(x; q,t) &:= \{z^{n-k-1}\} \left[ \sum_{P \in \mathcal{LD}_n} q^{\mathrm{dinv}(P)} t^{\mathrm{area}(P)}
\prod_{i \in \mathrm{Val}(P)} \left( 1 + z/t^{d_i(P) + 1} \right) x^P \right].
\end{align}
The notation here reflects the combinatorics; $\Rise_{n,k}(x; q,t)$ is governed by rises  
of  Dyck paths and $\Val_{n,k}(x; q,t)$ is governed by valleys of  Dyck paths; see
\cite{HRW} for details.

When $k = n-1$, the Delta Conjecture specializes to the Shuffle Conjecture of
Haglund, Haiman, Remmel, Loehr, and Ulyanov \cite{HHRLU}.
The Shuffle Conjecture gives an expansion of the symmetric function $\Delta'_{e_{n-1}} e_n$ into 
the fundamental basis of quasisymmetric functions, and therefore an expansion of the bigraded
Frobenius character
of the module of diagonal harmonics into this quasisymmetric basis.
It is an open problem to find a nice representation theoretic interpretation of $\Delta'_{e_k} e_n$ for general 
$k < n$.

Carlsson and Mellit proved the Shuffle Conjecture using
 a dazzling and groundbreaking array of plethystic operators and combinatorial techniques \cite{CM}.
However, the Delta Conjecture remains open for general $k < n$.  Various special cases of the Delta 
Conjecture were proven in \cite{HRW}.  One of these special cases is as follows; it concerns setting one of 
the variables $q, t$ equal to zero.

\begin{theorem}
\label{delta-q-t-zero-hrw} (\cite[Theorem 4.1]{HRW})
The coefficients of the monomial quasisymmetric function $M_{1^n}$ are equal in each of the following:
\begin{equation*}
\Rise_{n,k}(x; q, 0), \hspace{0.1in} 
\Rise_{n,k}(x; 0,q), \hspace{0.1in} 
\Val_{n,k}(x; q,0), \hspace{0.1in} 
 \Delta'_{e_k}e_n|_{t = 0}, \hspace{0.1in} 
\Delta'_{e_k} e_n|_{q = 0, t = q}.
\end{equation*}
\end{theorem}

Not included in Theorem~\ref{delta-q-t-zero-hrw} is the expression $\Val(x; 0, q)$.  In Section~\ref{Words}
below, we will add it in by proving the following result.

\begin{theorem}
\label{valley-function-equality}
For all $n > k \geq 0$, we have the equality 
\begin{equation}
\Val_{n,k}(x; 0, q) = \Val_{n,k}(x; q, 0).
\end{equation}
Consequently, we can append $\Val(x; 0, q)$ to the list in Theorem~\ref{delta-q-t-zero-hrw}.
\end{theorem}

Theorem~\ref{valley-function-equality} proves a conjecture of A. T. Wilson
\cite[Conjecture 1]{WMultiset} as well as 
a related conjecture of Haglund, Remmel, and Wilson \cite[Conjecture 4.1]{HRW}.

\subsection{Permutation and Word Statistics}
Our proof of Theorem~\ref{valley-function-equality} will rely on equidistribution results for statistics
on ordered set partitions and ordered multiset partitions which generalize statistics on permutations and words.

Let $w = w_1 \dots w_m$ be any word in the alphabet $\{1, 2, \dots \}$.
The {\em inversion number} of $w$ is
\begin{equation}
\inv(\pi) := \# \{ (i < j) \,:\, w_i > w_j \}.
\end{equation}
A {\em descent} of $w$ is an index $1 \leq i \leq m-1$ such that $w_i > w_{i+1}$.  The 
{\em descent set} of $w$ is
\begin{equation}
\Des(w) := \{ 1 \leq i \leq m-1 \,:\, \text{$i$ is a descent of $w$} \}.
\end{equation}
Finally, the {\em major index} of $w$ is
\begin{equation}
\maj(w) := \sum_{i \in \Des(w)} i.
\end{equation}

The study of statistics on words (and in particular  permutations in the symmetric group $S_n$ on 
$[n] := \{1, 2, \dots, n\}$)
is a major facet of enumerative combinatorics.  One of the foundational results on permutations statistics is 
due to MacMahon \cite{M}; in 1915 he proved that 
\begin{equation}
\label{eq:macmahon}
\sum_{\pi \in S_n} q^{\inv(\pi)} = \sum_{\pi \in S_n} q^{\maj(\pi)} = [n]!_q,
\end{equation}
where $[n]!_q := [n]_q [n-1]_q \cdots [1]_q$ and $[r]_q := 1 + q + \cdots + q^{r-1}$.
In his honor, any permutation statistic which shares this distribution is called {\em Mahonian}.

In \cite{HRW} it was proven that $\Rise_{n,k}(x; q,t)$ and $\Val_{n,k}(x; q,t)$ at $q = 0$ or 
$t = 0$ can be written using generalizations of permutation and word statistics to 
statistics on ordered set partitions and ordered multiset partitions.  The result for $\Val_{n,k}(x; q, t)$ is as follows.

\begin{theorem}
\label{valley-interpretation} (\cite{HRW})
For all $n > k \geq 0$, we have the equalities
\begin{align}
\Val_{n,k}(x; q, 0) &= \sum_{\beta \models_0 n} \sum_{\mu \in \OP_{\beta, k+1}} q^{\inv(\mu)} x^{\beta}, \\
\Val_{n,k}(x; 0, q) &= \sum_{\beta \models_0 n} \sum_{\mu \in \OP_{\beta, k+1}} q^{\minimaj(\mu)} x^{\beta}.
\end{align}
\end{theorem}

The outer sums in Theorem~\ref{valley-interpretation} are over all weak compositions
$\beta = (\beta_1, \beta_2, \dots )$ of $n$.
The inner sums in Theorem~\ref{valley-interpretation} are over all $(k+1)$-block ordered multiset partitions 
$\mu$ of 
the multiset $\{1^{\beta_1}, 2^{\beta_2}, \dots \}$.  The entries $\inv(\mu)$
and $\minimaj(\mu)$ appearing in the exponents 
are generalizations to ordered multiset partitions of the inversion and major index statistic on words.
Theorem~\ref{valley-function-equality} is therefore equivalent to proving that 
$ \sum_{\mu \in \OP_{\beta, k}} q^{\inv(\mu)} =
 \sum_{\mu \in \OP_{\beta, k}} q^{\minimaj(\mu)}$ for 
 any $\beta \models_0 n$.  This will be demonstrated in Section~\ref{Words}.

The inversion statistic $\inv(\mu)$ on ordered multiset partitions was studied by Steingr\'imsson 
\cite{Stein} (under the name $\mathrm{ROS}$)  and has
received a fair amount of attention \cite{RWSet, WMultiset}.  
The minimaj statistic $\minimaj(\mu)$ is newer; it first appeared in \cite{HRW}. 
Given an ordered multiset partition $\mu$, one considers
the set of words $W(\mu)$ obtained by permuting letters within the blocks of $\mu$.  For example,
if $\mu = ( 13 | 1 | 23 )$, then 
$$W(\mu) = \{ 13123,  31123, 13132, 31132 \}.$$  
The number $\minimaj(\mu)$ is defined to be the minimum of 
the major index statistic over the set $W(\mu)$.  
In this example, we have 
$$\maj(13123) = 2, \text{ } \maj(31123) = 1, \text{ } \maj(13132) = 6, \text{ and } \maj(31132) = 5,$$
 so that
$\minimaj(13 | 1 | 23) = 1$.


\subsection{Organization}
The rest of the paper is organized as follows.

In {\bf Section~\ref{Permutations}} we will study the $\inv$ and $\minimaj$ statistics on ordered set partitions.
We will prove that they have the same distribution on ordered set partitions with a fixed shape,
and calculate this distribution $F_{n,\alpha}(q)$ explicitly.  The polynomials 
$F_{n,\alpha}(q)$ are nonstandard $q$-analogs of the multinomial coefficients
${n \choose \alpha}$.

In {\bf Section~\ref{Words}} we generalize from permutations and ordered set partitions to words and ordered
multiset partitions.  We will see that the theory of $\inv$ and $\minimaj$ behaves differently in the
presence of repeated letters.  In particular, these statistics are {\em not} equidistributed on 
ordered multiset partitions of the same weight and shape.  On the other hand, these statistics
{\em are} equidistributed on ordered multiset partitions of the same weight and with a fixed
number of blocks.

In {\bf Section~\ref{Colored}} we generalize in another direction to colored permutations and colored
ordered set partitions.  As in the uncolored case, we will see that the statistics 
$\inv$ and $\minimaj$ share the same distribution on these sets.  In fact, we will prove that this
distribution is a natural generalization $([r]_q)^n F_{n,\alpha}(q^r)$ 
of the polynomial $F_{n,\alpha}(q)$.  However, we will see that attempting to mimic the arguments
of Section~\ref{Permutations} in the presence of colors
 results in a seemingly difficult polynomial
identity (Proposition~\ref{polynomial-identity}).

We close in {\bf Section~\ref{Closing}} with some algebraic questions regarding the polynomial
$F_{n,\alpha}(q)$.

\section{Permutations and ordered set partitions}
\label{Permutations}

\subsection{Ordered Set Partitions}
A natural generalization of  permutations is given by ordered set partitions.  An {\em ordered set partition}
$\sigma$ of $[n]$ is a sequence $\sigma = (B_1, \dots, B_k)$ of nonempty subsets of $[n]$
(called {\em blocks}) such that $[n] = B_1 \uplus \dots \uplus B_k$.  
The sequence of cardinalities 
$(|B_1|, \dots, |B_k|)$ is called the {\em shape} of $\sigma$, so
an ordered set partition of shape $1^n$ is just a permutation in $S_n$.
For example, we have that 
\begin{equation*}
\sigma = ( \{2, 5\}, \{1\}, \{3, 4\} ) 
\end{equation*}
is an ordered set partition of $[5]$ with $3$ blocks
and shape $(2,1,2)$.  
We will adopt the shorthand of writing ordered set partitions 
$\sigma = (B_1 | \ldots | B_k)$
by using vertical bars to demarcate
blocks and writing entries within blocks in increasing order.  
Our example ordered set partition is therefore
\begin{equation*}
\sigma = (25 | 1 | 34).
\end{equation*}

Recall that a {\em weak composition of $n$} is a finite sequence 
$\alpha = (\alpha_1, \dots, \alpha_k)$ of nonnegative integers satisfying $\alpha_1 + \cdots + \alpha_k = n$.
We say that $\ell(\alpha) = k$.
A {\em composition of $n$} is a weak composition of $n$ without any zeros.  We use the notation
$\alpha \models_0 n$ to mean that $\alpha$ is a weak composition of $n$ and
$\alpha \models n$ to mean that $\alpha$ is a composition of $n$.  

For any $n \geq k$ and $\alpha \models n$, we introduce the following families of ordered set partitions:
\begin{align}
\OP_n &:= \{ \text{all ordered set partitions of $[n]$} \}, \\
\OP_{n,k} &:= \{ \text{all ordered set partitions of $[n]$ with $k$ blocks} \}, \\
\OP_{n, \alpha} &:= \{ \text{all ordered set partitions of $[n]$ of shape $\alpha$} \}.
\end{align}

Let $\sigma = (B_1 | \dots | B_k)$ be an ordered set partition.  An {\em inversion} of $\sigma$ is a pair 
$1 \leq i < j \leq n$ such that $j \in B_{\ell}$ and $i = \min(B_m)$ for some $\ell < m$.   The {\em inversion count} 
$\inv(\sigma)$ is the number of inversions
of $\sigma$.  
For example, the inversions
of $(25|1|34)$ are $(1,2), (1,5),$ and $(2,3)$ and we have
$\inv(25|1|34) = 3$.
The $\inv$ statistic on ordered set partitions was introduced by Steingr\'imsson \cite{Stein}.

For $\sigma \in \OP_n$,
 let $S(\sigma)$  be the set of permutations in $S_n$ obtained by rearranging letters within the blocks 
 of $\sigma$.  For example, we have
 \begin{equation*}
 S(25|1|34) = \{25134, 52134, 25143, 52143\}. 
 \end{equation*}
 As $\sigma$ varies over $\OP_n$, the sets
 $S(\sigma)$ are precisely the parabolic
 cosets of the symmetric group $S_n$.  As in Section~\ref{Introduction}, we set
 \begin{equation}
 \minimaj(\sigma) := \min \{ \maj(\pi) \,:\, \pi \in S(\sigma) \}.
 \end{equation}
 For example, we have 
 \begin{align*}
 \minimaj(25|1|34) &= \min \{ \maj(25134), \maj(52134), \maj(25143), \maj(52143) \} \\
 &= \min \{2, 3, 6, 7\} = 2.
 \end{align*}
 
For any $n \geq k$ and $\alpha \models n$,
 we introduce the  generating functions for $\inv$ and $\minimaj$.
 \begin{equation}
 \begin{cases}
 I_n(q) := \sum_{\sigma \in \OP_n} q^{\inv(\sigma)}, \\
  I_{n,k}(q) := \sum_{\sigma \in \OP_{n,k}} q^{\inv(\sigma)}, \\
    I_{n,\alpha}(q) := \sum_{\sigma \in \OP_{n,\alpha}} q^{\inv(\sigma)},
 \end{cases}
  \begin{cases}
 M_n(q) := \sum_{\sigma \in \OP_n} q^{\minimaj(\sigma)}, \\
  M_{n,k}(q) := \sum_{\sigma \in \OP_{n,k}} q^{\minimaj(\sigma)}, \\
    M_{n,\alpha}(q) := \sum_{\sigma \in \OP_{n,\alpha}} q^{\minimaj(\sigma)}.
 \end{cases}
 \end{equation}

\subsection{Segmented Permutations}
It will be convenient to consider permutations and words which carry a segmentation.
Formally, a {\em $k$-segmented permutation} of size $n$ is a pair $(\pi, \alpha)$ where $\pi \in S_n$ and 
$\alpha \models n$ has $\ell(\alpha) = k$.  
We write $(\pi, \alpha)$ by inserting dots between just after the positions
$\alpha_1, \alpha_1 + \alpha_2, \dots$ of $\pi$, resulting in a figure
$\pi = \pi[1] \cdot \pi[2] \cdot \ldots \cdot \pi[k]$.  The words $\pi[1], \pi[2], \dots, \pi[k]$ are called the {\em segments}
of $(\pi, \alpha)$.

For example, we have that 
\begin{equation*}
(\pi, \alpha) = (627318945, (3,1,3,2))
\end{equation*}
is a $4$-segmented permutation of size $9$.  We represent
this pair as 
\begin{equation*}
\pi[1] \cdot \pi[2] \cdot \pi[3] \cdot \pi[4] = 627 \cdot 3 \cdot 189 \cdot 45 
\end{equation*}
and we have the segments
\begin{equation*}
\pi[1] = 627, \pi[2] = 3, \pi[3] = 189, \pi[4] = 45.
\end{equation*}

For any permutation statistic $\mathrm{stat}$, we have a natural extension of
$\mathrm{stat}$ to segmented permutations given by forgetting the segmentation; that is
\begin{equation*}
\mathrm{stat}(\pi, \alpha) := \mathrm{stat}(\pi).  
\end{equation*}
The role played by segmentation will be  minor in this
section, but segmentation will become very important in Section~\ref{Words} when we allow
repetition of letters.

Segmented permutations can be used to give an alternative characterization of $\minimaj$.  In particular, we 
define a map 
\begin{equation}
\pi: \OP_{n, k} \longrightarrow \{ \text{$k$-segmented permutations of size $n$} \}
\end{equation}
 as follows.  
 
 \begin{defn}
 \label{segment-pi-map}
 Let $\sigma = (B_1 | \dots | B_k) \in \OP_{n,k}$.  We define 
 $\pi(\sigma) = \pi[1] \cdot \pi[2] \cdot \ldots \cdot \pi[k]$ by the following recursive procedure.
 For $1 \leq i \leq k$, write $B_i = \{j^{(i)}_1 < j^{(i)}_2 < \cdots < j^{(i)}_{\alpha_i} \}.$

 \begin{itemize}
 \item  Let the $k^{th}$ segment $\pi[k]$ be the letters of $B_k$ written in increasing order:
\begin{equation*}
 \pi[k] := j_1^{(k)} j_2^{(k)} \cdots j_{\alpha_k}^{(k)}.
 \end{equation*}
 \item  For $1 \leq i \leq k-1$, assume that the segment $\pi[i+1]$ has been defined.  Let $r$ be the first 
 letter in $\pi[i+1]$ and let $0 \leq m \leq a_i$ be the maximum index such that $j_m^{(i)} \leq r$ (where we 
 set $j_0^{(i)} := - \infty$).  Define the $i^{th}$ segment $\pi[i]$ to be the sequence
 \begin{equation*}
 \pi[i] := j_{m+1}^{(i)} j_{m+2}^{(i)} \cdots j_{\alpha_i}^{(i)} j_1^{(i)} j_2^{(i)} \cdots j_m^{(i)}.
 \end{equation*}
 \end{itemize}
 \end{defn}

 For example, we have  
 \begin{equation*}
 \pi(257|6|148|39) = 725 \cdot 6 \cdot 481 \cdot 39.
 \end{equation*}
 Observe that if $\pi(\sigma) = (\pi, \alpha)$, we have that $\pi \in S(\sigma)$.
 
 The following alternative formulation of $\minimaj$ is easy to see; its proof is left to the reader.
 
 \begin{lemma}
 \label{minimaj-alternative}
 Let $\sigma \in \OP_n$ be an ordered set partition of $[n]$ and let $\pi(\sigma) = (\pi, \alpha)$.
 The major index statistic $\maj$ achieves a unique minimum value on $S(\sigma)$, and this value is achieved 
 at $\pi$.  In particular,
 we have 
 \begin{equation*}
 \minimaj(\sigma) = \maj(\pi(\sigma)).
 \end{equation*}
 \end{lemma}
 
 For example, we may calculate 
 \begin{equation*}
 \minimaj(257|6|148|39) = \maj(725 \cdot 6 \cdot 481 \cdot 39) = 1 + 4 + 6 = 11.
 \end{equation*}
 
 Given any sequence $(C_1 | \dots | C_k)$ of nonempty, finite, and pairwise disjoint sets of integers, define
 the {\em standardization}
 $\overline{(C_1 | \dots | C_k)}$ to be the unique ordered set partition of 
 $[|C_1| + \cdots + |C_k|]$ which is order isomorphic to $(C_1 | \dots | C_k)$.  
 For example, we have
 \begin{equation*}
 \overline{( 2 9 | 5 | 3 7)} = (1 5 | 3 | 2 4).  
 \end{equation*}
 The $\minimaj$ statistic satisfies the following
 `compression' property. 
 
 \begin{lemma}
 \label{compress}
 Let $\sigma = (B_1 | \dots | B_{k-1} | B_k)$ be an ordered set partition.  We have
 \begin{equation}
 \minimaj(\sigma) = \minimaj( \overline{B_1 | \dots | B_{k-1} | \min(B_k)} ).
 \end{equation}
 \end{lemma}
 
 \begin{proof}
 We have that $\pi(\overline{B_1 | \dots | B_{k-1} | \min(B_k)})$ coincides with 
 $\pi(\sigma) = \pi[1] \cdot \pi[2] \cdot \dots \cdot \pi[k]$, except that 
 its terminal segment $\pi[k]$ is shortened to its first letter.  
Moreover, the segment $\pi[k]$ does not contain any descents.  Now apply 
Lemma~\ref{minimaj-alternative}.
 \end{proof}
 
For example, we have
 \begin{equation*}
 \minimaj(13|58|2467) = \minimaj(\overline{13|58|2}) = \minimaj(13|45|2).
 \end{equation*}

\subsection{$M_{n, \alpha}(q) = I_{n, \alpha}(q) = F_{n, \alpha}(q)$.}
For a composition $(\alpha_1, \dots, \alpha_k) \models n$ with
$\ell(\alpha) = k$, let $F_{n,\alpha}(q)$ be the following polynomial:
\begin{equation}
F_{n,\alpha}(q) 
 := \prod_{i = 1}^k {\alpha_i - 1 \choose \alpha_i - 1} + {\alpha_i \choose \alpha_i - 1} q
 + {\alpha_i + 1 \choose \alpha_i - 1} q^2 + \cdots + 
 {\alpha_1 + \cdots + \alpha_i - 1 \choose \alpha_i - 1} q^{\alpha_1 + \cdots + \alpha_{i -1} }. 
\end{equation}
For example, if $\alpha = (2,3,2) \models 7$ we have
\begin{align*}
F_{7,(2,3,2)}(q) &= \left[ {2 \choose 2} + {3 \choose 2}q + {4 \choose 2}q^2\right] \times
\left[ {1 \choose 1} + {2 \choose 1} q + \cdots + {6 \choose 1} q^5 \right].
\end{align*}

At the extremes, we have that 
$F_{n,(n)}(q) = 1$ and
$F_{n,(1^n)}(q) = [n]!_q$.  We have that $F_{n,\alpha}(1)$ is the multinomial coefficient
${n \choose \alpha} = {n \choose \alpha_1, \dots, \alpha_k}$, but for general $\alpha$ the polynomial
$F_{n,\alpha}(q)$ is {\em not} the standard $q$-multinomial
${n \brack \alpha_1, \dots, \alpha_k}_q = \frac{[n]!_q}{[\alpha_1]!_q \cdots [\alpha_k]!_q}$.

In this subsection we will prove that $F_{n, \alpha}(q)$ coincides with 
both $M_{n, \alpha}(q)$ and $I_{n, \alpha}(q)$.  The strategy is to show that these three polynomials
obey the same recursion. 

Let $\ZZ_n = \langle c \rangle$ act on permutations in $S_n$ by decrementing every letter by
$1$ modulo $n$.  
For example, if $n = 8$ we have  
\begin{equation*}
c.(35278416) = 24167385.
\end{equation*}
The following simple lemma will be crucial.

\begin{lemma}
\label{cycle-major-permutations}
Let $\pi = \pi_1 \pi_2 \dots \pi_n \in S_n$ and assume $\pi_n \neq 1$.  We have 
\begin{equation}
\maj(c.\pi) = \maj(\pi) + 1.
\end{equation}
\end{lemma}

\begin{proof}
The action of $c$ on $\pi$ preserves the descent set, except that the descent just before $1$ moves
a single unit to the right.
\end{proof}

Decrementing by $1$ modulo $n$ also
defines an action of $\ZZ_n$ on ordered set partitions of $[n]$.
Moreover, we have an action of $\ZZ_n$ on segmented permutations given by preserving segmentation:
\begin{equation*}
c.(\pi, \alpha) := (c.\pi, \alpha). 
\end{equation*} 

Recall the map $\pi$ of Definition~\ref{segment-pi-map} from ordered set partitions to segmented words.
Although $\pi$ does not commute with the action of $\ZZ_n$ in general, we have commutativity
when we restrict to ordered set partitions whose terminal block size is $1$.

\begin{lemma}
\label{cycle-commute-permutations}
Let $\alpha = (\alpha_1, \dots, \alpha_k) \models n$ and suppose that $\alpha_k = 1$.
Let $\sigma \in \OP_{n, \alpha}$.  We have that
\begin{equation}
c.\pi(\sigma) = \pi(c.\sigma).
\end{equation}
\end{lemma}

\begin{proof}
Let $B_{i_0}$ be the block of $\sigma$ containing $1$.
Write the segments of $\pi(\sigma)$ and $\pi(c.\sigma)$ as 
$\pi(\sigma) = \pi[1] \cdot \pi[2] \cdot \ldots \cdot \pi[k]$ and
$\pi(c.\sigma) = \pi'[1] \cdot \pi'[2] \cdot \ldots \cdot \pi'[k]$.  We want to show that
$\pi[i] = c.\pi'[i]$ for all $i$.  This is true when $i = k$ because $\alpha_k = 1$.  
Definition~\ref{segment-pi-map}
 makes it clear that 
$\pi[i] = c.\pi'[i]$ for $i > i_0$.

 If 
$\pi[i_0] = j_{m+1} j_{m+2} \dots j_{\alpha_{i_0}} 1 j_2 \dots j_m$, we see that
$$\pi'[i_0] = (j_{m+1}-1) (j_{m+2}-1) \dots (j_{\alpha_{i_0}}-1) n (j_2 -1)\dots (j_m-1) = c.\pi[i_0].$$

Now consider the segments $\pi[i_0 - 1]$ and $\pi'[i_0 - 1]$.  If $m < \alpha_{i_0}$ above, then the first letter of 
$\pi'[i_0]$ is the first letter of $\pi[i_0]$, less one.  Definition~\ref{segment-pi-map} shows that
$\pi'[i_0 - 1] = c.\pi[i_0 - 1]$ in this case.  
If $m = \alpha_{i_0}$ above, then the first letter of $\pi[i_0]$ is $1$ and the first letter of $\pi'[i_0]$ is $n$.
Another application of Definition~\ref{segment-pi-map} shows that 
$\pi'[i_0 - 1] = c.\pi[i_0 - 1]$ in this case, as well.

Finally, the equality $\pi'[i_0 - 1] = c.\pi[i_0 - 1]$ and Definition~\ref{segment-pi-map} make it clear
 that $\pi'[i] = c.\pi[i]$ for all $i < i_0$.
\end{proof}

We are ready to prove our first equidistribution result.

\begin{theorem}
\label{same-distribution-permutations}
Let $\alpha$ be a composition of $n$.  We have that
\begin{equation}
\label{main-equation}
M_{n, \alpha}(q) = I_{n,\alpha}(q) = F_{n,\alpha}(q).
\end{equation}
\end{theorem}

\begin{proof}
The $F$-polynomials satisfy the recursion
\begin{equation}
\label{F-recursion}
F_{n, (\alpha_1, \dots, \alpha_{k-1}, \alpha_k)}(q) = 
F_{n - \alpha_k, (\alpha_1, \dots, \alpha_{k-1})}(q) \times 
\left[ {\alpha_k - 1 \choose \alpha_k - 1} + {\alpha_k \choose \alpha_k - 1} q + \cdots + {n-1 \choose \alpha_k - 1} q^{n - \alpha_k}\right].
\end{equation}

It is easy to see 
that $I_{n, \alpha}(q)$ satisfies the recursion in Equation~\ref{F-recursion}.  Given an ordered set partition
$\sigma = (B_1 | \dots | B_k)$ of shape $(\alpha_1, \dots, \alpha_k)$,
 the block $B_k$ is involved in $n - \min(B_k)$ inversions.  After selecting a number
$1 \leq j \leq n$ for  $\min(B_k)$, we have  ${n-j \choose \alpha_k - 1}$
 choices to complete the block $B_k$.  It follows
that $I_{n, \alpha}(q) = F_{n, \alpha}(q)$ for all $\alpha$.

To prove that $F_{n,\alpha}(q) = M_{n,\alpha}(q)$,
 write $\alpha = (\alpha_1, \dots, \alpha_{k-1}, \alpha_k)$ and
assume inductively 
that $F_{n - \alpha_k, (\alpha_1, \dots, \alpha_{k-1}}(q) = M_{n - \alpha_k, (\alpha_1, \dots, \alpha_{k-1})}(q)$. 
For any ordered set partition
$\sigma$ of shape $(\alpha_1, \dots, \alpha_{k-1})$, let $(\sigma | n-\alpha_k+1)$ be the ordered set partition of type 
$(\alpha_1, \dots, \alpha_{k-1}, 1)$ obtained by appending the singleton $\{n-\alpha_k+1\}$ to the end of $\sigma$.
Clearly we have $\pi(\sigma | n - \alpha_k + 1) = \pi(\sigma) \cdot (n-\alpha_k+1)$, so that by 
Lemma~\ref{minimaj-alternative},
\begin{equation}
\minimaj(\sigma | n - \alpha_k + 1) = \maj(\pi(\sigma) \cdot (n-\alpha_k+1)) = \maj(\pi(\sigma)) = \minimaj(\sigma).
\end{equation}
Lemmas~\ref{minimaj-alternative}, 
\ref{cycle-commute-permutations}, and \ref{cycle-major-permutations} imply that for $0 \leq d \leq n - \alpha_k$,
\begin{equation}
\minimaj(c^d.(\sigma | n - \alpha_k + 1)) = \minimaj(\sigma) + d.
\end{equation}
This gives the recursion in Equation~\ref{F-recursion} for  $M_{n, \alpha}(q)$ when $\alpha_k = 1$.  For general 
$\alpha_k \geq 1$,
we apply Lemma~\ref{compress} to see that for any $d$, there are precisely 
${n-1 \choose \alpha_k - d + 1}$ ordered set partitions $\sigma'$ of $[n]$ of type $\alpha$ satisfying 
$\minimaj(\sigma') = \minimaj(\sigma) + d$.  This shows that the recursion of
Equation~\ref{F-recursion} holds for 
$M_{n, \alpha}(q)$.
\end{proof}

Although our proof of Theorem~\ref{same-distribution-permutations} is combinatorial, it is not bijective.  
J. Remmel and A. T. Wilson are preparing a bijective proof of Theorem~\ref{same-distribution-permutations}.

If we coarsen Theorem~\ref{same-distribution-permutations} to consider ordered set partitions
with $k$ blocks, we get the following corollary.  

\begin{corollary}
\label{same-distribution-stirling}
Let $n \geq k > 0$.  We have that
\begin{equation}
\label{main-equation}
M_{n, k}(q) = I_{n, k} (q) = [k]!_q \Stir_{n,k}(q),
\end{equation}
where the $q$-Stirling number $\Stir_{n,k}(q)$ is defined by the recursion
\begin{equation}
\Stir_{n,k}(q) = \Stir_{n-1,k-1}(q) + [k]_q \Stir_{n-1,k}(q)
\end{equation}
and the initial condition $\Stir_{0,k}(q) := \delta_{0,k}$.
\end{corollary}

\begin{proof}
Steingr\'imsson \cite[Theorem 4]{Stein} proved that 
$I_{n,k}(q) = [k]!_q \Stir_{n,k}(q)$ (Steingr\'imsson uses $\mathrm{ROS}$
for our $\mathrm{inv}$).  Now apply Theorem~\ref{same-distribution-permutations}.
\end{proof}

In \cite{RWSet}, Remmel and Wilson prove bijectively that $\mathrm{inv}$ is equidistributed on 
$\OP_{n,k}$ with a different $\mathrm{maj}$-like statistic on ordered set partitions 
due to Steingr\'imsson \cite{Stein}.  However, the statistics considered in \cite{RWSet} are not 
equidistributed on the finer sets $\OP_{n, \alpha}$.

Relating Corollary~\ref{same-distribution-stirling} to the Delta Conjecture yields the following result.

\begin{corollary}
\label{valley-specializations}
The coefficients of the monomial quasisymmetric function $M_{(1^n)}$ in 
the monomial expansions of 
$\Val_{n,k}(x; q,0)$ and $\Val_{n,k}(x; 0,q)$ are equal.
\end{corollary}

Corollary~\ref{valley-specializations} proves
a conjecture of Haglund, Remmel, and Wilson
\cite[Conjecture 4.1]{HRW} 
and allows us to add
$\Val_{n,k}(x;0,q)$ to the list in 
Theorem~\ref{delta-q-t-zero-hrw}.

In the next section we will strengthen Corollary~\ref{valley-specializations}
by proving 
$\Val_{n,k}(x;q,0) = \Val_{n,k}(x;0,q)$.  If the Delta Conjecture is true,
we would have the stronger symmetry 
$\Val_{n,k}(x; q,t) = \Val_{n,k}(x; t,q)$.

\section{Words and ordered multiset partitions}
\label{Words}

\subsection{Ordered Multiset Partitions}  
Many definitions and results on ordered set partitions and permutations generalize naturally
to ordered multiset partitions and words.
However, we will discover a significant divergence in the behavior of the statistics $\inv$ and $\minimaj$.

For any weak composition $\beta = (\beta_1, \beta_2, \dots )$, an
{\em ordered multiset partition of weight $\beta$} is a sequence 
$\mu = (B_1 | \cdots | B_k)$ of nonempty {\bf sets} such that 
$\{1^{\beta_1}, 2^{\beta_2}, \cdots \}$ is the multiset union $B_1 \cup \dots \cup B_k$.  
The total number of letters is called the {\em size} of $\mu$,
the sets $B_i$ are called the {\em blocks} of $\mu$, and the composition
$\alpha = (|B_1|, \dots, |B_k|)$ of set cardinalities is called the {\em shape} of $\mu$.

For example, we have that 
\begin{equation*}
\mu = (245 | 134 | 2457 | 4) 
\end{equation*}
is an ordered multiset partition with $4$ blocks of size $11$, weight 
$\beta = (1,2,1,4,2,0,1)$, and shape $(3,3,4,1)$.  Although letters can be repeated in an ordered multiset
partition, no single block of an ordered multiset partition can have repeated letters.

Given integers $n \geq k$, a weak composition $\beta \models_0 n$, and
a composition $\alpha \models n$, we introduce the following families of ordered multiset partitions:
\begin{align}
\OP_{\beta} &:= \{ \text{all ordered multiset partitions of weight $\beta$} \}, \\
\OP_{\beta,k} &:= \{ \text{all ordered multiset partitions of weight $\beta$ with $k$ blocks} \}, \\
\OP_{\beta, \alpha} &:= \{ \text{all ordered multiset partitions of weight $\beta$ of shape $\alpha$} \}.
\end{align}

Let $\mu = (B_1 | \dots | B_k) \in \OP_{\beta, \alpha}$.  
As in the case of ordered set partitions,
an {\em inversion} in $\mu$ is a pair $i < j$ such that 
$i = \min(B_m)$ and $j \in B_{\ell}$ for some $\ell < m$.  
The statistic $\inv(\mu)$ counts the number of inversions of $\mu$.
If $\mu$ is the ordered multiset partition shown above, we have
$\inv(\mu) = 10$.

Given a weak composition $\beta \models n$, let $W_{\beta}$ denote the set of all words 
$w_1 \dots w_n$ containing $\beta_i$ copies of $i$, for all $i$.  
For any ordered multiset partition $\mu \in \OP_{\beta, \alpha}$,
let $W(\mu)$ be the set of words in $W_{\beta}$ obtained by rearranging the letters of $\mu$ within its blocks.  
As in Section~\ref{Introduction}, the  
$\minimaj$ statistic on ordered multiset partitions is 
\begin{equation}
\minimaj(\mu) = \min \{ \maj(w) \,:\, w \in W(\mu) \}.
\end{equation}

As in Section~\ref{Permutations}, we introduce the following generating functions for 
$\inv$ and $\minimaj$ on ordered multiset partitions.
 \begin{equation}
 \begin{cases}
 I_{\beta}(q) := \sum_{\mu \in \OP_{\beta}} q^{\inv(\mu)}, \\
  I_{\beta,k}(q) := \sum_{\mu \in \OP_{\beta,k}} q^{\inv(\mu)}, \\
    I_{\beta,\alpha}(q) := \sum_{\mu \in \OP_{n,\alpha}} q^{\inv(\mu)},
 \end{cases}
  \begin{cases}
 M_{\beta}(q) := \sum_{\mu \in \OP_{\beta}} q^{\minimaj(\mu)}, \\
  M_{\beta,k}(q) := \sum_{\mu \in \OP_{\beta,k}} q^{\minimaj(\mu)}, \\
    M_{\beta,\alpha}(q) := \sum_{\mu \in \OP_{\beta,\alpha}} q^{\minimaj(\mu)}.
 \end{cases}
 \end{equation}
 
 In light of Theorem~\ref{same-distribution-permutations}, the reader may guess that we have
 $I_{\beta, \alpha}(q) = M_{\beta, \alpha}(q)$ for any weak composition $\beta$ and any 
 composition $\alpha$ with $|\beta| = |\alpha|$.  However, this statement is false.  For example,
 consider $\beta = (2,2,1)$ and $\alpha = (2,1,2)$.  We have
 \begin{equation*}
 \OP_{(2,2,1),(2,1,2)} = \{ (12|3|12), (12|1|23), (23|1|12), (12|2|13), (13|2|12) \}.
 \end{equation*}
 We leave it for the reader to check that
 \begin{equation*}
 I_{(2,2,1),(2,1,2)}(q) = q + 2q^2 + q^3 + q^4
 \end{equation*}
 whereas
 \begin{equation*}
 M_{(2,2,1),(2,1,2)}(q) = q + q^2 + 2q^3 + q^4.
 \end{equation*}
 Despite this, we will prove the coarser equality
 $I_{\beta, k}(q) = M_{\beta,k}(q)$ for any $\beta \models_0 n$ and any $k \leq n$.

\subsection{Segmented Words}
Given a weak composition $\beta \models_0 n$, let $W_{\beta}$ be the collection of words
$w_1 \dots w_n$ with $\beta_i$ copies of the letter $i$ for all $i$.  In particular, if $\beta = (1^n)$ 
we get $W_{(1^n)} = S_n$.

A {\em $k$-segmented word} is a pair $(w, \alpha)$, where $w = w_1 \dots w_n$ is a word and 
$\alpha \models n$ is a composition with $\ell(\alpha) = k$.  
We represent a $k$-segmented word $(w, \alpha)$ as 
$w = w[1] \cdot \ldots \cdot w[k]$, were dots are placed after positions 
$\alpha_1, \alpha_1 + \alpha_2, \dots$ of the word $w$.
The subwords $w[1], w[2], \dots, w[k]$ are called {\em segments}.
If $\mathrm{stat}$ is any statistic on words, we get a statistic on segmented words by
$\mathrm{stat}(w, \alpha) := \mathrm{stat}(w)$.

Given any $k$-segmented word $w = w[1] \cdot \ldots \cdot w[k]$, we have a natural ordered 
multiset partition with $k$ blocks obtained by replacing dots with bars and deleting repeated letters.
For example, the ordered multiset partition corresponding to
\begin{equation*}
243 \cdot 2 \cdot 114 \cdot 34
\end{equation*} 
is 
\begin{equation*}
(243 | 2 | 14 | 34). 
\end{equation*} 
In particular, the length of a segmented word could be larger than the size of the associated
ordered multiset partition.
The length of a segmented word equals the size of the corresponding 
ordered multiset partition if and only if none of the segments have repeated letters.

Definition~\ref{segment-pi-map} can be repeated verbatim to get a function
\begin{equation}
w: \OP_{\beta, k} \longrightarrow \{\text{$k$-segmented words} \}.
\end{equation}
For example, we have
\begin{equation*}
w(124|2|13|245|34) = 412 \cdot 2 \cdot 13 \cdot 452 \cdot 34.
\end{equation*}
 Lemma~\ref{minimaj-alternative} remains true in the presence of repeated letters.
 
 \begin{lemma}
 \label{minimaj-alternative-word}
 Let $\mu$ be an ordered multiset partition with $w(\mu) = (w, \alpha)$.  
 The major index statistic $\maj$ achieves a unique minimum value
 on $W(\mu)$, and this value is achieved at $w$. In particular, we have
 $\minimaj(\mu) = \maj(w(\mu))$.
\end{lemma}

For example, we have
\begin{equation*}
\minimaj(124|2|13|245|34) = \maj(412 \cdot 2 \cdot 13 \cdot 452 \cdot 34) 
= 1 + 4 + 8 = 13.
\end{equation*}

\subsection{Recursions for $\inv$ and $\minimaj$}
In this subsection we develop recursions for $\inv$ and $\minimaj$ on ordered multiset partitions.
We will see that these recursions do not coincide, which explains why the generating 
functions for these statistics disagree on the sets $\OP_{\beta, \alpha}$.

Let $\alpha = (\alpha_1, \dots, \alpha_k) \models n$ be a composition and let
$\beta = (\beta_1, \dots, \beta_m) \models_0 n$ be a weak composition.
For any subset $S \subseteq [m]$, let $\chi(S) = (\chi(S)_1, \dots, \chi(S)_m)$ 
be the length $m$ indicator vector defined by
\begin{equation}
\chi(S)_i = \begin{cases}
1 & i \in S, \\
0 & i \notin S.
\end{cases}
\end{equation}
The recursion for $\inv$ on ordered multiset partitions is as follows.

\begin{lemma}
\label{inv-recursion-word}
Preserve the notation from above.  The generating function $I_{\beta, \alpha}(q)$ satisfies the following formula.
\begin{equation}
I_{\beta, \alpha}(q) = \sum_S q^{\sum_{i = \min(S)+1}^m (\beta_i - \chi(S)_i)} I_{\beta - \chi(S), (\alpha_1, \dots, \alpha_{k-1})}(q),
\end{equation}
where the sum is over all subsets $S \subseteq [m]$ such that $|S| = \alpha_k$ and $\beta_i > 0$ for all 
$i \in S$.
\end{lemma}

\begin{proof}
Consider a typical ordered multiset partition $\mu = (B_1 | \dots | B_{k-1} | B_k) \in \OP_{\beta, \alpha}$.  Writing 
$S = B_k$,
we have that $(B_1 | \dots | B_{k-1}) \in \OP_{\beta - \chi(S), (\alpha_1, \dots, \alpha_{k-1})}$.  Moreover, the addition of 
$S$ to the end of this ordered multiset partition contributes one inversion for each element of the sets $B_i$ 
($1 \leq i \leq k-1$) which is $> \min(S)$.
\end{proof}

The recursion for $\minimaj$ is different and a bit more involved to prove.
Our starting point is the following generalization of Lemma~\ref{cycle-major-permutations} from $S_n$ to
 to $W_{\beta}$ for $\beta \models_0 n$ with $\ell(\beta) = m$.  
 As before, we let $\ZZ_m = \langle c \rangle$ act on $W_{\beta}$ by
 decrementing all letters by $1$ modulo $m$.
 
\begin{lemma}
\label{cycle-major-words}
Let $w = w_1 \dots w_n \in W_{\beta}$ where $\ell(\beta) = m$
and  $w_n \neq 1$.  Then $\maj(c.w) = \maj(w) + \beta_1$.
\end{lemma}

\begin{proof}
The map $c$ moves every descent occurring just before a maximal contiguous run of $1$'s in $w$ to the position
at the end of this run.  
\end{proof}

Lemma~\ref{cycle-commute-permutations} generalizes to the setting of repeated letters as follows.

\begin{lemma}
\label{cycle-commute-words}
Let $\alpha = (\alpha_1, \dots, \alpha_k) \models n$ be a composition and let
$\beta = (\beta_1, \dots, \beta_m) \models_0 n$ be a weak composition.  Assume that $\alpha_k = 1$.  We have that
$w(c.\mu) = c.w(\mu)$ for any $\mu \in \OP_{\beta,\alpha}$.
\end{lemma}

\begin{proof}
Similar to the proof of Lemma~\ref{cycle-commute-permutations}.  Use Lemma~\ref{minimaj-alternative-word} instead of 
Lemma~\ref{minimaj-alternative}.
\end{proof}

Since removing letters from an ordered multiset partition gives an ordered multiset partition (provided that none of the blocks
are emptied), there is no need to introduce standardization in the analog of Lemma~\ref{compress} for 
ordered multiset partitions.

\begin{lemma}
\label{compress-word}
Let $(B_1 | \dots | B_{k-1} | B_k)$ be an ordered multiset partition.  We have 
\begin{equation}
\minimaj(B_1 | \dots | B_{k-1} | B_k) = \minimaj(B_1 | \dots | B_{k-1} | \min(B_k)).
\end{equation}
\end{lemma}

\begin{proof}
Similar to the proof of Lemma~\ref{compress}.  
Use Lemma~\ref{minimaj-alternative-word} instead of Lemma~\ref{minimaj-alternative}.
\end{proof}

In light of Lemma~\ref{compress-word}, it is enough to state a recursion for $M_{\beta,\alpha}(q)$ when the last 
part of $\alpha$ equals $1$.

\begin{lemma}
\label{basic-minimaj-recursion-word}
Let $\alpha = (\alpha_1, \dots, \alpha_k) \models n$ be a composition with $\alpha_k = 1$ and let
$\beta = (\beta_1, \dots, \beta_m) \models_0 n$ be a weak composition.  We have
\begin{equation}
M_{\beta, \alpha}(q) = \sum_{\beta_i > 0} q^{\beta_{i+1} + \beta_{i+2} + \cdots + \beta_m} 
M_{(\beta_{i+1}, \dots, \beta_m, \beta_1, \dots, \beta_i - 1), (\alpha_1, \dots, \alpha_{k-1})}(q).
\end{equation}
\end{lemma}

\begin{proof}
Without loss of generality, we may assume that $\beta \models n$ is in fact a strict composition.  By construction, we have that
any ordered multiset partition of the form $(\mu | m) \in \OP_{\beta, \alpha}$ satisfies
$w(\mu | m) = w(\mu) \cdot m$, so that Lemma~\ref{minimaj-alternative-word} gives
\begin{equation}
\minimaj(\mu | m) = \minimaj(\mu).
\end{equation}
On the other hand, for $1 \leq i \leq m-1$, Lemma~\ref{cycle-commute-words} and 
Lemma~\ref{cycle-major-words} show that
for a typical $\mu' \in \OP_{(\beta_1, \dots, \beta_{m-i} - 1, \dots, \beta_m), (\alpha_1, \dots, \alpha_{k-1})}$, we have

\begin{align}
\minimaj(\mu' | m-i) &= \minimaj(c^{i}.( c^{-i}.\mu' | m))  \\
&= \minimaj(c^{-i}.\mu') - \beta_{m-i} - \cdots - \beta_m.
\end{align}

The corresponding identity of generating functions follows.
\end{proof}

Observe that the weight composition $\beta$ gets cycled in 
the recursion of Lemma~\ref{basic-minimaj-recursion-word},
but that no such cycling appears in Lemma~\ref{inv-recursion-word}.
This is irrelevant in the permutation case $\beta = (1^n)$, but in the case of words
we get different behavior of $\inv$ and $\minimaj$ on $\OP_{\beta,\alpha}$.

\subsection{Switch Maps} 
In order to prove that $\inv$ and $\minimaj$ are equidistributed on $\OP_{\beta, k}$, we will
show first that $\minimaj$ has the same distribution on $\OP_{\beta, k}$ and $\OP_{\beta', k}$
whenever $\beta'$ is a rearrangement of $\beta$.  Equivalently, we will show that the 
quasisymmetric function $\Val_{n,k}(x;0,q)$ is symmetric in the $x$ variables.
We will do this using a family of involutions on ordered multiset partitions 
called {\em switch maps}.  As the general definition of switch maps is a bit involved,
let us first examine the special case of switch maps defined on words.

Let 
$\beta = (\beta_1, \beta_2, \dots, \beta_m) \models_0 n$ and let $\beta' \models_0 n$ be a weak 
composition whose entries are a rearrangement of those of $\beta$.
Our goal is to give a combinatorial proof that $\maj$ has the same distribution on the set of words
$W_{\beta}$ as on the set of words $W_{\beta'}$.

It is enough to consider the case $\beta' = s_i.\beta := (\beta_1, \dots, \beta_{i+1}, \beta_i, \dots, \beta_m)$
for some $1 \leq i \leq m-1$.  We define a map $t_i: W_{\beta} \rightarrow W_{s_i.\beta}$ as follows.
Let $w = w_1 \dots w_n \in W_{\beta}$.
Underline every consecutive pair of letters $(i+1)i$ which appears in $w$.  Now overline all maximal
runs of not-underlined letters which are of the form $i^a (i+1)^b$ for some $a, b \geq 0$.  
Let $t_i(w)$ be the word obtained by replacing every overlined run $i^a (i+1)^b$ with the run
$i^b (i+1)^a$.
For example, we have
\begin{equation*}
t_3(122\underline{43}\overline{34}2\overline{3}22\overline{4}1\overline{344}) = 
122\underline{43}\overline{34}2\overline{4}22\overline{3}1\overline{334}.
\end{equation*}

\begin{proposition}
\label{bender-knuth-words}
Let $\beta = (\beta_1, \dots, \beta_m) \models_0 n$.  For $1 \leq i \leq m-1$, the map 
$t_i: W_{\beta} \rightarrow W_{s_i.\beta}$ is bijective and preserves descent sets, and hence major index.
\end{proposition}

\begin{proof}
The $t_i$ are involutions, and so bijections. By construction, the  
$t_i$ preserve descent sets.
\end{proof}

We conclude that $\maj$ has the same distribution on $W_{\beta}$ and $W_{\beta'}$ whenever
$\beta'$ is a rearrangement of $\beta$.  The operators $\{t_i \,:\, 1 \leq i \leq m-1\}$ satisfy the relations:
\begin{equation}
\label{switch-relations}
\begin{cases}
t_i^2 = 1 & \text{for all $1 \leq i \leq m-1$,} \\
t_i t_j = t_j t_i & \text{for $|i - j| > 1$.}
\end{cases}
\end{equation}

The main task of this subsection is to generalize the $t_i$ maps from words to ordered multiset partitions.
In particular, for any weak composition $\beta = (\beta_1, \dots, \beta_m) \models_0 n$, any
$1 \leq i \leq m-1$, and any 
$k > 0$, we will define a {\em switch map}
$t_i: \OP_{\beta, k} \rightarrow \OP_{s_i.\beta, k}$.  The switch maps
will satisfy the relations of Equation~\ref{switch-relations}.
In particular, the relation $t_i^2 = 1$ will force the switch maps to be bijective.
Moreover, we will have that 
$\Des(w(t_i(\mu))) = \Des(w(\mu))$ for any ordered multiset partition $\mu$.
Lemma~\ref{minimaj-alternative-word}  therefore implies that
$\minimaj(t_i(\mu)) = \minimaj(\mu)$.
In the all singletons case $\OP_{\beta, n}$, the switch map $t_i$ will just be the map defined above.  
However, in the general case the switch map $t_i$ will not preserve shape
(and it {\bf cannot}; see the example following Corollary~\ref{valley-symmetric} below).

In order to define the $t_i$ operators on ordered multiset partitions, 
we need to carefully consider the relative positions of the letters $i$ and $i+1$ in
segmented words lying in the image of the map
$w: \OP_{\beta, k} \rightarrow \{k-$segmented words$\}$.
The following trichotomy will allow us to fruitfully generalize the notion of a contiguous pair of letters 
of the form $(i+1)i$.

\begin{lemma}
\label{trichotomy}
Let $\alpha = (\alpha_1, \dots, \alpha_k) \models n$ be a composition and let $\mu = (B_1 | \dots | B_k)$ be an ordered
multiset partition of shape $\alpha$.  Fix a letter $i$ and consider the $k$-segmented word
$w(\mu) = w[1] \cdot \ldots \cdot w[k]$.  For each $1 \leq \ell \leq k$, exactly one of the following statements
holds for the segment $w[\ell]$.
\begin{enumerate}
\item  The segment $w[\ell]$ contains at most one letter from the set $\{i, i+1\}$.
\item  The letters $i$ and $i+1$ both appear in $w[\ell]$,  are adjacent, and occur in the order $i(i+1)$.
\item  The letters $i$ and $i+1$ both appear in $w[\ell]$, the letter $i+1$ is the first letter of $w[\ell]$, and
the letter $i$ is the last letter of $\ell$.  Moreover, we have that $\ell < k$ and the first letter of $w[\ell+1]$ is $i$.
\end{enumerate}
\end{lemma}

\begin{proof}
If $\ell = k$, the segment $w[k]$ consists of the letters in $B_k$ written in increasing order, 
and the claimed trichotomy holds.
Assume that $1 \leq \ell \leq k-1$ and this trichotomy
 holds for the segment $w[\ell+1]$.  Let $r$ be the first letter of $w[\ell+1]$
and write $B_{\ell} = \{j^{(\ell)}_1 < \dots < j^{(\ell)}_{\alpha_{\ell}} \}$. 
 By the definition of the map $w$, the segment $w[\ell]$ is
$j^{(\ell)}_{m+1} \dots j^{(\ell)}_{\alpha_{\ell}} j^{(\ell)}_1 \dots j^{(\ell)}_m$, where $m$ is maximal such that 
$j^{(\ell)}_m \leq r$.  

If $i, i+1 \in B_{\ell}$, we  have that $i (i+1)$ is a contiguous subword 
of  $j^{(\ell)}_{m+1} \dots j^{(\ell)}_{\alpha_{\ell}} j^{(\ell)}_1 \dots j^{(\ell)}_m$ unless $r = i$.
If $r = i$ and $i, i+1 \in B_{\ell}$, we get that $j^{(\ell)}_{m+1} = i+1$ and $j^{(\ell)}_m = i$.  
We deduce that the desired trichotomy holds for the segment $w[\ell]$.
\end{proof}

In defining the $t_i$ operator on words, we froze all consecutive pairs of letters of the form
$(i+1)i$.  To extend the $t_i$ operators to ordered multiset partitions, we will need to expand the
collection of subwords that get frozen.  We capture this idea in the following definition of an
{\em $i$-drop}.  The reader should compare this definition to Lemma~\ref{trichotomy}.

\begin{defn}
\label{i-drop-def}
{\em ($i$-drops)}
Let $w = w[1] \cdot \ldots \cdot w[k]$ be a $k$-segmented word and let $i$ be a letter.  An
{\em $i$-drop} in $w$ is a contiguous subword of the form
$(i+1) j_1 \dots j_{\ell} i$, where
\begin{itemize}
\item  none of the letters $j_1, \dots, j_{\ell}$ are equal to $i$ or $i+1$, and
\item either $(i+1) j_1 \dots j_{\ell} i$ or $(i+1) j_1 \dots j_{\ell}$ is a segment of $w[1] \cdot \ldots \cdot w[k]$.
\end{itemize}
\end{defn}

We will consider $i$-drops in $k$-segmented 
words $w[1] \cdot \ldots \cdot w[k]$ arising from ordered multiset partitions.
For example, the $11$-segmented word  which is the image under $w$ of the ordered multiset partition
$(1237|34|4|3|3|4|3467|3|3|457|356)$ is shown below with all $3$-drops underlined:
\begin{equation*}
7123 \cdot 34 \cdot \underline{4 \cdot 3} \cdot 3 \cdot 4 \cdot \underline{4673} \cdot 3 \cdot 3 \cdot \underline{457  \cdot 3}56.
\end{equation*}
If $k = n$ is the length of $w$, we recover the classical notion of a contiguous pair $(i+1)i$ in a word $w$.
We next describe how to underline and overline the segmented word $w(\mu)$ arising from an
ordered multiset partition $\mu$.

\begin{defn}
\label{word-decoration}
{\em (Word Decoration)}
Let $\beta = (\beta_1, \dots, \beta_m) \models_0 n$, let $1 \leq i \leq m-1$, and let $k > 0$.  

Let $\mu \in \OP_{\beta,k}$ and write $w(\mu) = w[1] \cdot \ldots \cdot w[k]$.  The 
{\em $i$-decorated $k$-segmented word} corresponding to $\mu$ is obtained 
from $w[1] \cdot \ldots \cdot w[k]$ as follows.
\begin{itemize}
\item  Underline every $i$-drop in $w(\mu)$, and then
\item  overline every maximal contiguous subword of $w(\mu)$ consisting entirely of $i$'s and $(i+1)$'s which are not underlined.
\end{itemize}
\end{defn}

Continuing our running example, the $3$-decorated $11$-segmented word corresponding to the ordered 
multiset partition
$(1237|34|4|3|3|4|3467|3|3|457|356)$ is
\begin{equation*}
712\overline{3 \cdot 34} \cdot \underline{4 \cdot 3} \cdot \overline{3 \cdot 4} 
\cdot \underline{4673} \cdot \overline{3 \cdot 3} \cdot \underline{457  \cdot 3}56.
\end{equation*}
In general, the overlined subwords in decorated words all have the same form.

\begin{lemma}
\label{overlined-form}
Every overlined subword in an 
$i$-decorated $k$-segmented word has the form $i^a (i+1)^b$ for some $a, b \geq 0$. 
\end{lemma}

\begin{proof}
If not, then we can find an overlined subword with a contiguous pair of letters equal to $(i+1)i$.  If these letters are in the same segment 
of $w(\mu) = w[1] \cdot \ldots \cdot w[k]$, 
by Lemma~\ref{trichotomy} this segment must equal $(i+1)i$,
which would make 
$(i+1)i$ an (underlined) $i$-drop, a contradiction. 
 The letters $(i+1)$ and $i$ are therefore in different segments, forcing $i$ to be the first letter of its segment and
 $(i+1)$ to be the last letter of its segment.
In turn, the definition of $w$ forces the segment of $w(\mu)$ containing $i+1$ to have length 
$1$, so that $(i+1)i$ is again an $i$-drop,
which is again a contradiction.
\end{proof} 

We define the ordered set partition $t_i(\mu)$ in terms of its $k$-segmented word 
$w(t_i(\mu))$.  The word $w'$ underlying the $k$-segmented word $w(t_i(\mu)) = w'[1] \cdot \ldots \cdot w'[k]$ is easy to describe:
for every overlined subword in the $i$-decorated $k$-segmented word corresponding to
$w(\mu)$, replace $i^a (i+1)^b$ with $i^b (i+1)^a$.  Obtaining the correct segmentation
$w'[1] \cdot \ldots \cdot w'[k]$ from the $k$-segmented word $w[1] \cdot \ldots \cdot w[k]$ is more involved and depends on the local structure
of $w(\mu)$ near each overlined subword.  

\begin{defn}
\label{switch-maps}
{\em (Switch Maps)}
Let $\beta = (\beta_1, \dots, \beta_m) \models_0 n$, let $1 \leq i \leq m-1$, and let $k > 0$.  We define 
an ordered set partition $t_i(\mu)$ as follows.

Let $\mu \in \OP_{\beta, k}$ be an ordered multiset partition.
Form the $i$-decorated $k$-segmented word $w(\mu) = w[1] \cdot \ldots \cdot w[k]$ corresponding to $\mu$.

We define a new $k$-segmented word $w' = w'[1] \cdot \ldots \cdot w'[k]$ from this decorated word by altering the letters and segmentation
in $w(\mu)$ near each overlined subword of $w(\mu)$ as follows.  Let $i^a (i+1)^b$ be a typical overlined subword of $w(\mu)$.  
\begin{itemize}
\item  {\bf Case 1:} {\em $a > 0$ and $b = 0$.}  If the overlined run $i^a$ is not immediately preceded by an $i$-drop in $w(\mu)$, then
trade every $i$ in this run for an $i+1$ and leave the segmentation at this run unchanged in $w'$.  If the overlined run $i^a$
is immediately preceded by an $i$-drop in $w(\mu)$, then the structure of $w(\mu)$ near this run must look like
\begin{equation*}
\cdot \underline{(i+1) j_1 \dots j_{\ell} i} \cdot \underbrace{ \overline{i \cdot  \ldots \cdot i} }_a \text{ or } 
\cdot \underline{(i+1) j_1 \dots j_{\ell} \cdot i} \cdot \underbrace{ \overline{i \cdot  \ldots \cdot i}}_a.
\end{equation*}
These expressions differ in a single dot before the leftmost $i$.

Trade these in $w'$ for
\begin{equation*}
\cdot (i+1) j_1 \dots j_{\ell}  \cdot i \underbrace{(i+1) \cdot  \ldots \cdot (i+1)}_a \text{ or } 
\cdot (i+1) j_1 \dots j_{\ell} \cdot i \cdot \underbrace{(i+1) \cdot  \ldots \cdot (i+1)}_a,
\end{equation*}
respectively.
\item {\bf Case 2:} {\em $a = 0$ and $b > 0$.}  If the overlined run $(i+1)^b$ is not immediately preceded by an $i$-drop in $w(\mu)$, then
trade every $i+1$ in this run for an $i$ and leave the segmentation of this run unchanged in $w'$.  If the overlined run $(i+1)^b$ is 
immediately preceded by an $i$-drop in $w(\mu)$, then the structure of $w(\mu)$ near this run must look like
\begin{equation*}
\cdot \underline{(i+1) j_1 \dots j_{\ell}  \cdot i} \underbrace{ \overline{(i+1) \cdot  \ldots \cdot (i+1)} }_b \text{ or } 
\cdot \underline{(i+1) j_1 \dots j_{\ell} \cdot i} \cdot \underbrace{ \overline{(i+1) \cdot  \ldots \cdot (i+1)}}_b.
\end{equation*}
Trade these in $w'$ for
\begin{equation*}
\cdot (i+1) j_1 \dots j_{\ell} i \cdot \underbrace{i \cdot  \ldots \cdot i}_b \text{ or } 
\cdot (i+1) j_1 \dots j_{\ell} \cdot i \cdot \underbrace{i \cdot  \ldots \cdot i}_b,
\end{equation*}
respectively.
\item  {\bf Case 3:} {\em $a > 0$ and $b > 0$.}  The structure of $w(\mu)$ near the run $i^a (i+1)^b$ must look like
\begin{equation*}
\underbrace{ \overline{i \cdot \ldots \cdot i} }_a \underbrace{ \overline{(i+1) \cdot \ldots \cdot (i+1)} }_b \text{ or }
\underbrace{ \overline{i \cdot \ldots \cdot i}}_a \cdot \underbrace{ \overline{(i+1) \cdot \ldots \cdot (i+1)}}_b.
\end{equation*}
Trade  these in $w'$ for 
\begin{equation*}
\underbrace{ i \cdot \ldots \cdot i}_b \underbrace{(i+1) \cdot \ldots \cdot (i+1)}_a \text{ or }
\underbrace{ i \cdot \ldots \cdot i}_b \cdot \underbrace{(i+1) \cdot \ldots \cdot (i+1)}_a
\end{equation*}
respectively.
\end{itemize}

After applying the above procedure to every overlined run in $w(\mu)$, we get a $k$-segmented word 
$w' = w'[1] \cdot \ldots \cdot w'[k]$.  Define $t_i(\mu) = (B'_1 | \dots | B'_k)$ to be the ordered multiset partition
obtained by letting the elements in the block $B'_{\ell}$ be the letters in the segment $w'[\ell]$.
\end{defn}

The operator $t_i$ is called a {\em switch map} on ordered multiset partitions.
As Definition~\ref{switch-maps} is elaborate, let us give an example of applying switch maps to ordered
multiset partitions.

\begin{example}
We let 
$\mu = (1237|34|4|3|3|4|3467|3|3|457|356)$ and $i = 3$.  We wish to compute 
$t_3(\mu)$.
To do this, we first compute the $3$-decorated $11$-segmented word corresponding to $\mu$ to get 
\begin{equation*}
w(\mu) = w[1] \cdot \ldots \cdot w[11] = 712\overline{3 \cdot 34} \cdot \underline{4 \cdot 3} \cdot \overline{3 \cdot 4} 
\cdot \underline{4673} \cdot \overline{3 \cdot 3} \cdot \underline{457  \cdot 3}56.
\end{equation*}
Now we alter the overlined subwords and the segmentation near them to get
$w' = w'[1] \cdot \ldots \cdot w'[11]$.  Reading off from 
Definition~\ref{switch-maps}, we get 
\begin{equation*}
w' = w'[1] \cdot \ldots \cdot w'[11] =
71234 \cdot 4 \cdot 4 \cdot 3 \cdot 3 \cdot 4
\cdot 467 \cdot 3 4 \cdot 4 \cdot 457  \cdot 356.
\end{equation*}
Replacing segments with blocks, we get
\begin{equation*}
t_3(1237|34|4|3|3|4|3467|3|3|457|356) = (12347 | 4 | 4 | 3 | 3 | 4 | 467 | 34 | 4 | 457 | 346).
\end{equation*}

Observe that we have
\begin{equation*}
w(12347 | 4 | 4 | 3 | 3 | 4 | 467 | 34 | 4 | 457 | 346) = 
71234 \cdot 4 \cdot 4 \cdot 3 \cdot 3 \cdot 4
\cdot 467 \cdot 3 4 \cdot 4 \cdot 457  \cdot 356.
\end{equation*}
That is, we have $w(t_3(\mu))$ is precisely the $11$-segmented word 
$w' = w'[1] \cdot \ldots \cdot w'[11]$
constructed in
Definition~\ref{switch-maps}.  Moreover, the $3$-decoration of $w(t_3(\mu))$ is given by
\begin{equation*}
712\overline{34 \cdot 4} \cdot \underline{4 \cdot 3} \cdot \overline{3 \cdot 4} 
\cdot \underline{467 \cdot 3} \overline{4 \cdot 4} \cdot \underline{457  \cdot 3}56.
\end{equation*}
In particular, the positions of the underlined and overlined portions of this word (ignoring segmentation)
are precisely the same as the corresponding positions in the $3$-decorated $11$-segmented word
$w(\mu) = w[1] \cdot \ldots \cdot w[11]$.  If we apply the transformations of Definition~\ref{switch-maps}
to $w(t_3(\mu))$ we get
\begin{equation*}
7123 \cdot 34 \cdot 4 \cdot 3 \cdot 3 \cdot 4
\cdot 4673 \cdot 3 \cdot 3 \cdot 457  \cdot 356.
\end{equation*}
This is precisely the $11$-segmented word $w(\mu)$.  In particular, we have that 
$t_3^2(\mu) = \mu$.
\end{example}

The key facts about switch maps are collected in the following proposition.

\begin{proposition}
\label{switch-well-defined}
Let $\beta = (\beta_1, \dots, \beta_m) \models_0 n$, let $1 \leq i \leq m-1$, and let $k > 0$.
Let $\mu \in \OP_{\beta, k}$.

\begin{enumerate}
\item
We have that $t_i(\mu)$ is an
ordered multiset partition in $\OP_{s_i.\beta, k}$.  
The action of $t_i$ therefore gives a map
\begin{equation}
t_i: \OP_{\beta, k} \longrightarrow \OP_{s_i.\beta,k}.
\end{equation}
\item
If $w' = w'[1] \cdot \ldots \cdot w'[k]$
is the $k$-segmented word constructed in Definition~\ref{switch-maps}, we have that
$w(t_i(\mu)) = w'[1] \cdot \ldots \cdot w'[k]$ as $k$-segmented words.  
\item
We have that $\Des(w(t_i(\mu))) = \Des(w(\mu))$.
\item
We have that $\minimaj(t_i(\mu)) = \minimaj(t_i(\mu))$.
\item
We have that $t_i^2(\mu) = \mu$.
\end{enumerate}
\end{proposition}

\begin{proof}
(1)  We first check that $t_i(\mu)$ has the same size as $\mu$.  This is equivalent to showing that
none of the segments in the $k$-segmented word $w' = w'[1] \cdot \ldots \cdot w'[k]$ of 
Definition~\ref{switch-maps} has repeated letters.   

To begin, notice that none of the segments of $w(\mu) = w[1] \cdot \ldots \cdot w[k]$
have repeated letters.  Consider forming the $i$-decorated $k$-segmented word corresponding
to $w(\mu)$ and let $i^a (i+1)^b$ be a typical overlined subword.  We need to show that none of the 
transformations in Definition~\ref{switch-maps} induce repeated letters.  

\begin{itemize}
\item  {\bf Case 1:}  {\em $a > 0$ and $b = 0$.}  If the overlined run $i^a$ is not immediately preceded
by an $i$-drop, then 
\begin{itemize}
\item  the segment of $w(\mu)$ containing the leftmost $i$ in this run cannot contain $(i+1)$, for otherwise
this segment would be an (underlined) $i$-drop by Lemma~\ref{trichotomy}, and
\item  the segment of $w(\mu)$ containing the rightmost $i$ in this run cannot contain $(i+1)$, for otherwise
Lemma~\ref{trichotomy} implies that this rightmost $i$ must be immediately followed by $(i+1)$ and
this $(i+1)$ cannot be the first letter in an $i$-drop, contradicting the maximality of overlined runs. 
\end{itemize}
We conclude that changing every $i$ to an $i+1$ and leaving segmentation preserved does not induce 
any repeated letters within segments.

If the overlined run $i^a$ is immediately preceded by an $i$-drop, then we apply either the left or
right branch transformations in Case 1 of Definition~\ref{switch-maps}.  In either situation, the reasoning
of the last paragraph shows that the segment of $w(\mu)$ containing the rightmost $i$ cannot also contain
an $(i+1)$.  It follows that neither of these transformations induce repeated letters within segments.
\item  {\bf Case 2:}  {\em $a = 0$ and $b > 0$.}  If the overlined run $(i+1)^b$ is not immediately preceded
by an $i$-drop, then
\begin{itemize}
\item  the segment containing the leftmost $(i+1)$ in this run cannot also contain an $i$, for otherwise
this segment would end in $i(i+1)$ by Lemma~\ref{trichotomy}.  Lemma~\ref{trichotomy}
also implies that this $i$ cannot be the end of an $i$-drop, contradicting the maximality of overlined
runs, and
\item the segment containing the rightmost $(i+1)$ cannot also contain an $i$, for otherwise this segment
would be an (underlined) $i$-drop by Lemma~\ref{trichotomy}.
\end{itemize}
We conclude that changing every $i+1$ to an $i$ and leaving segmentation preserved does not 
induce any repeated letters within segments.

If the overlined run $(i+1)^b$ is immediately preceded by an $i$-drop, then we apply either the left
or right branch transformations of Case 2 of Definition~\ref{switch-maps}.  The reasoning of the last
paragraph shows that the segment containing the rightmost $(i+1)$ cannot also contain an $i$.
Moreover, the definition of an $i$-drop shows that none of the letters $j_1, \dots, j_{\ell}$
appearing in these braces equals $i$.   We conclude that neither of the transformations of these
branches induces repeated letters within segments.
\item  {\bf Case 3:}  {\em $a, b > 0$.}  In this case we apply the either the left or right brach of 
Case 3 of Definition~\ref{switch-maps} to the overlined run $i^a (i+1)^b$.  Since $a, b > 0$ it is evident
that this transformation does not induce repeated letters within blocks.
\end{itemize}

We conclude that none of the segments in the transformed word $w' = w'[1] \cdot \ldots \cdot w'[k]$
contains repeated letters, so that $t_i(\mu)$ has the same size as $\mu$.  
It remains to show that $t_i(\mu) \in \OP_{s_i.\beta,k}$.  Equivalently, we  need to show that
$t_i(\mu)$ has $k$ blocks and weight $s_i.\beta$.

Since none of the transformations of Definition~\ref{switch-maps}
changes the number of segments of a word (just the lengths of these segments), we 
have that $t_i(\mu)$ has the same number of blocks as $\mu$.  That is,
the ordered multiset  partition $t_i(\mu)$ has $k$ blocks.  

Since each 
underlined $i$-drop in $w(\mu)$ contains 
precisely one $i$ and one $(i+1)$, and each of the overlined $i^a (i+1)^b$ runs is transformed to
$i^b (i+1)^a$ by Definition~\ref{switch-maps}, we conclude that 
$t_i(\mu)$ has weight $s_i.\beta$.  
In summary, we have that $t_i(\mu)$ is an ordered
multiset partition in $\OP_{s_i.\beta,k}$.

(2)  Write $w(t_i(\mu)) := w''[1] \cdot \ldots \cdot w''[k]$ and $t_i(\mu) = (B'_1 | \cdots | B'_k)$.  
We need to show that $w'[s] = w''[s]$ for all
$1 \leq s \leq k$.  Definition~\ref{switch-maps} makes it clear that each segment $w''[s]$ is a rearrangement
of the letters in the segment $w'[s]$. 
It is therefore enough to show that the segments of 
$w'[1] \cdot \ldots \cdot w'[k]$ satisfy the recursive definition of the map $w$ applied
to $(B'_1 | \cdots | B'_k)$.  

Consider a typical overlined run $i^a (i+1)^b$ in the $i$-decorated $k$-segmented word corresponding to
$w$.  We may inductively assume that the segmented words $w'$ and $w''$ agree to the right of 
this overlined word.

\begin{itemize}
\item  {\bf Case 1:}  {\em $a > 0$ and $b = 0$.}  The portion of the segmented word $w(\mu)$ immediately 
to the right of this overlined run looks like 
\begin{equation*}
\underbrace{ \overline{i \cdot \ldots \cdot i} }_a j'_1 \cdots j'_{\ell'} \cdot r,
\end{equation*}
where we could have $\ell' = 0$.  By the definition of $w$, we have that $r \neq i$.  If
$r = i+1$, then the rightmost $i$ would be part of an $i$-drop.  We conclude that 
$r > i+1$ or $r < i$.  The definition of $w$ and the maximality
of overlined subwords imply that none of the letters
$j'_1, \dots, j'_{\ell'}$ lie in $\{i, i+1\}$.  
Moreover, the segment containing $i$ does not contain $(i+1)$ be Lemma~\ref{trichotomy}.
The transformations of Case 1 change this portion of $w(\mu)$ to
\begin{equation*}
\underbrace{ (i+1) \cdot \ldots \cdot (i+1) }_a j'_1 \cdots j'_{\ell'} \cdot r,
\end{equation*}
perhaps altering the segmentation just to the left.  This segmented
subword agrees with the recursive definition of the map $w$.

\item  {\bf Case 2:}  {\em $a = 0$ and $b > 0$.}
The portion of the segmented word $w(\mu)$ immediately to the right of this overlined run looks like
\begin{equation*}
\underbrace{ \overline{(i+1) \cdot \ldots \cdot (i+1)} }_b j'_1 \cdots j'_{\ell'} \cdot r,
\end{equation*}
where we could have $\ell' = 0$. 
If $r = i$, the rightmost $(i+1)$ would be part of an $i$-drop.  By the definition of the map 
$w$, we know that $r \neq i$.  We conclude that $r > i+1$ or $r < i$.  Lemma~\ref{trichotomy}
implies that none of the letters $j'_1, \dots, j'_{\ell'}$ lie in $\{i, i+1\}$.
The  transformations of Case 2 change this portion of $w(\mu)$ to
\begin{equation*}
\underbrace{ i \cdot \ldots \cdot i }_b j'_1 \cdots j'_{\ell'} \cdot r,
\end{equation*}
perhaps altering the segmentation just to the left.  This segmented subword agrees with 
the recursive definition of the map $w$.

\item  {\bf Case 3:}  {\em $a, b > 0$.}  This case is easier and left to the reader.
\end{itemize}

(3)  A glance at the transformations of Definition~\ref{switch-maps} shows that none of them
affects descent sets.  Now apply (2).

(4)  Apply (3) together with Lemma~\ref{minimaj-alternative-word}.

(5)  Consider the $i$-decorated $k$-segmented
word $w(\mu) = w[1] \cdot \ldots \cdot w[k]$ corresponding to $\mu$.  
We consider applying the transformations
of Definition~\ref{switch-maps} to $w(\mu)$ twice.

Let $w' = w'[1] \cdot \ldots \cdot w'[k]$ be the $k$-segmented word constructed
in Definition~\ref{switch-maps} and consider the $i$-decoration of $w'$.  It is easy to see
that the underlined and overlined portions of this $i$-decoration occur in precisely the same
positions as those in $w(\mu)$.

The transformation in the left  branch of Case 1 leads to the transformation in the left 
branch of Case 2, and vice versa.  The transformation in the right  branch of Case 1 leads to the 
transformation of the right  branch of Case 2, and vice versa.  The transformation of  either
branch of Case 3 leads to another transformation of the same branch.  Since the roles of $a$ and $b$ are 
interchanged in the second transformation, the application to two transformations recovers the 
$k$-segmented word $w(\mu)$.  In particular, we have $t_i^2(\mu) = \mu$.
\end{proof}

It can  be checked that the switch maps satisfy 
$t_i t_j = t_j t_i$ for $|i - j| > 1$, but we will not need this relation here.
The most important consequence of Proposition~\ref{switch-well-defined} for us is the following.

\begin{corollary}
\label{rearrange}
Let $k > 0$, let $\beta$ be a weak composition, 
and let $\beta'$ be any weak composition obtained by rearranging the 
 parts of $\beta$.  We have that
\begin{equation}
M_{\beta, k}(q) = M_{\beta',k}(q).
\end{equation}
\end{corollary}

\begin{proof}
For any $i$, the map $t_i: \OP_{\beta, k} \rightarrow \OP_{s_i.\beta, k}$ is a $\minimaj$-preserving bijection.
\end{proof}

As another immediate corollary of Proposition~\ref{switch-well-defined}, we can deduce the symmetry
of $\Val_{n,k}(x;0,q)$ in the $x$ variables.  It is unknown whether $\Val_{n,k}(x; q,t)$ is symmetric.

\begin{corollary}
\label{valley-symmetric}
Let $n \geq k > 0$.  The formal power series
\begin{equation*}
\Val_{n,k}(x;0,q) =  \sum_{\mu} q^{\minimaj(\mu)} x^{\mathrm{weight}(\mu)}
\end{equation*}
is symmetric in the variables $x = (x_1, x_2, \dots )$, where the sum is over 
all ordered multiset partitions $\mu$ with $k+1$ blocks and size $n$.
\end{corollary}

Corollary~\ref{valley-symmetric} is false if we instead sum over ordered set partitions $\mu$ with fixed shape.
For example, we have 
\begin{equation*}
\minimaj(13|12) = 1, \text{ } \minimaj(12|13) = 1, \text{ }
\minimaj(12|23) = 0, \text{ } \minimaj(23|12) = 2.
\end{equation*}
It follows that the coefficient of $x_1^2 x_2 x_3$ in the quasisymmetric function
\begin{equation*}
 \sum_{\mathrm{shape}(\mu) = (2,2)} q^{\minimaj(\mu)} x^{\mathrm{weight}(\mu)}
 \end{equation*}
is $2q$ but the coefficient of $x_1 x_2 x_3^2$ is $1 + q^2$.

Proposition~\ref{switch-well-defined} can be used to prove something more refined than 
Corollary~\ref{valley-symmetric}.  Namely, for any subset $S \subseteq [n-1]$ the quasisymmetric
function
\begin{equation*}
F_{n,k,S}(x; q) := 
\sum_{\substack{\text{$\mu$ has size $n$} \\ \text{$\mu$ has $k$ blocks} \\ \Des(w(\mu)) = S}}
q^{\minimaj(\mu)} x^{\mathrm{weight}(\mu)} = q^{\sum_{i \in S} i} \times
\sum_{\substack{\text{$\mu$ has size $n$} \\ \text{$\mu$ has $k$ blocks} \\ \Des(w(\mu)) = S}}
 x^{\mathrm{weight}(\mu)} 
\end{equation*}
is symmetric.  I. Gessel observed (personal communication) that when $k = n$, the symmetry
of $F_{n,n,S}(x;q)$ follows from the fact that
\begin{equation*}
\sum_{\substack{\text{$\mu$ has size $n$} \\ \text{$\mu$ has $n$ blocks} \\ 
\Des(\mu) \subseteq \{i_1 < \cdots < i_r\}} } x^{\mathrm{weight}(\mu)} =
h_{i_1}(x) h_{i_2 - i_1}(x) \cdots h_{i_r - i_{r-1}}(x)
\end{equation*}
(where $h_d(x)$ is the complete homogeneous symmetric function of degree $d$) and  
the Principle of Inclusion-Exclusion \cite{GesselComm}.  The author is unaware of an
analogous proof of the symmetry of $F_{n,k,S}(x; q)$.
The author is also unaware of what relationship (if any) the polynomials 
$F_{n,k,S}(x;q)$ have with the Delta Conjecture.

\subsection{$M_{\beta,k}(q) = I_{\beta,k}(q)$}
We are ready to prove our equidistribution result for ordered multiset partitions.

\begin{theorem}
\label{same-distribution-word}
Let $\beta \models_0 n$ be a weak composition and let $k \in \ZZ_{\geq 0}$.  We have that
\begin{equation}
I_{\beta, k}(q) = M_{\beta, k}(q).
\end{equation}
\end{theorem}

\begin{proof}
Write $\beta = (\beta_1, \dots, \beta_m) \models_0 n$ and let
$\alpha = (\alpha_1, \dots, \alpha_k) \models n$ with $\alpha_k = 1$.  Lemma~\ref{basic-minimaj-recursion-word} says that
\begin{equation}
\label{same-distribution-word-first}
M_{\beta, \alpha}(q) = \sum_{\beta_i > 0} q^{\beta_{i+1} + \beta_{i+2} + \cdots + \beta_m} 
M_{(\beta_{i+1}, \dots, \beta_m, \beta_1, \dots, \beta_i - 1), (\alpha_1, \dots, \alpha_{k-1})}(q).
\end{equation}
Sum both sides of  
Equation~\ref{same-distribution-word-first}
over all compositions $\overline{\alpha} = (\alpha_1, \dots, \alpha_{k-1})$ of $n-1$ which have $k-1$ parts.
This yields
\begin{align}
\label{main-chain}
\sum_{\overline{\alpha}} M_{\beta, (\alpha_1, \dots, \alpha_{k-1}, 1)}(q) &=
\sum_{\overline{\alpha}} 
\sum_{\beta_i > 0} q^{\beta_{i+1} + \beta_{i+2} + \cdots + \beta_m} 
M_{(\beta_{i+1}, \dots, \beta_m, \beta_1, \dots, \beta_i - 1), \overline{\alpha}}(q) \\
&= \sum_{\beta_i > 0}
\sum_{\overline{\alpha}} q^{\beta_{i+1} + \beta_{i+2} + \cdots + \beta_m} 
M_{(\beta_{i+1}, \dots, \beta_m, \beta_1, \dots, \beta_i - 1), \overline{\alpha}}(q) \\
&= \sum_{\beta_i > 0}
q^{\beta_{i+1} + \beta_{i+2} + \cdots + \beta_m} 
M_{(\beta_{i+1}, \dots, \beta_m, \beta_1, \dots, \beta_i - 1), k-1}(q) \\
&= \sum_{\beta_i > 0} q^{\beta_{i+1} + \beta_{i+2} + \cdots + \beta_m} 
M_{(\beta_1, \dots, \beta_i - 1, \dots, \beta_m), k-1}(q).
\end{align}
The first equality is Equation~\ref{same-distribution-word-first} summed over all compositions
$\overline{\alpha} \models n-1$ with $k-1$ parts.
The second equality interchanges the order of summation.
The third equality evaluates the inner sum over all $\overline{\alpha}$.  The fourth equality is an application
of Corollary~\ref{rearrange}.

The left hand side of Equation~\ref{main-chain}
is the generating function for $\minimaj$ on the collection of all $k$-block multiset partitions 
of weight $\beta$ whose last block has size $1$.  Lemma~\ref{compress-word} shows that for general 
$\alpha_k \geq 1$,
\begin{equation}
\label{main-word-finale}
\sum_{(\alpha_1, \dots, \alpha_{k-1})}  
M_{\beta, (\alpha_1, \dots, \alpha_{k-1}, \alpha_k)}(q) = \sum_S q^{\sum_{i = \min(S)+1}^m (\beta_i - \chi(S)_i)} M_{\beta - \chi(S), k-1}(q),
\end{equation}
where the sum on the left hand side is over all $(k-1)$-part compositions
$(\alpha_1, \dots, \alpha_{k-1}) \models n - \alpha_k$ and  
 the external sum on the right hand side is over all subsets 
$S \subseteq [m]$ of indices such that $\beta_i > 0$ for all
$i \in S$.  Lemma~\ref{inv-recursion-word} shows that 
$\sum_{(\alpha_1, \dots, \alpha_{k-1})} I_{\beta, (\alpha_1, \dots, \alpha_{k-1}, \alpha_k)}(q)$ satisfies the same recursion.  
\end{proof}  

The proof of Theorem~\ref{same-distribution-word} proves a slightly stronger statement.  Namely, for any fixed 
weak composition $\beta$ and positive integer $a$, we have that 
\begin{equation}
\sum_{\mu} q^{\inv(\mu)} = \sum_{\mu} q^{\minimaj(\mu)},
\end{equation}
where the sums are over all $k$-block ordered multiset partitions $\mu$ of weight $\beta$ whose 
last block has size $a$.  The connection of this more refined equidistribution result to the Delta Conjecture
is unclear.

Theorem~\ref{valley-function-equality} follows by combining 
Theorems~\ref{valley-interpretation} and
\ref{same-distribution-word}.  We restate it here.

\noindent
{\bf Theorem 1.3.}  {\em For all $n > k \geq 0$, we have the equality
\begin{equation*}
\Val_{n,k}(x; 0, q) = \Val_{n,k}(x; q, 0).
\end{equation*}
Consequently, the coefficients of the monomial quasisymmetric function 
$M_{1^n}$ are equal in each of the following:
\begin{equation*}
\Rise_{n,k}(x;q,0), \text{ } \Rise_{n,k}(x; 0,q), \text{ } \Val_{n,k}(x; q,0), \text{ } \Val_{n,k}(x; 0, q), \text{ }
\Delta'_{e_k} e_n|_{t = 0}, \text{ } \Delta'_{e_k} e_n|_{q = 0, t = q}.
\end{equation*}}

In light of the Delta Conjecture, it is natural to ask for the expansion of the symmetric functions
$\Val_{n,k}(x;q,0) = \Val_{n,k}(x;0,q)$ in the Schur basis $\{s_{\lambda} \,:\, \lambda \vdash n\}$.
It turns out that $\Val_{n,k}(x; 0,q)$ is Schur positive -- the coefficients in this expansion 
are polynomials in $q$ with nonnegative coefficients.  An explicit formula
for these coefficients follows immediately from the work of  Wilson \cite{WMultiset}.

To give the Schur expansion of $\Val_{n,k}(x; 0,q)$, we will need some notation.  
Given a partition $\lambda \vdash n$, let $\mathrm{SYT}(\lambda)$ denote the set of standard
Young tableaux of shape $\lambda$.  For $T \in \mathrm{SYT}(\lambda)$, a {\em descent} of $T$
is an index $1 \leq i \leq n-1$ such that $i+1$ appears in a lower row of $T$ than $i$.  Let 
$\mathrm{des}(T)$ denote the number of descents of $T$.  Also let $\maj(T)$ denote the sum of the descents of $T$.

For example, suppose that $T \in \mathrm{SYT}(3,3,2)$ is the standard tableau
\begin{center}
\begin{Young}
1 & 3 & 4  \cr
2 & 6 & 7 \cr
5 & 8
\end{Young}.
\end{center}
The descents of $T$ are $1, 4,$ and $7$ so that $\mathrm{des}(T) = 3$ and
$\maj(T) = 1 + 4 + 7 = 12$.

The following corollary follows immediately from an insertion argument due to
Wilson \cite{WMultiset} and Theorem~\ref{valley-function-equality}.
We use the usual $q$-binomial ${n \brack k}_q := \frac{[n]!_q}{[k]!_q [n-k]!_q}$.

\begin{corollary} (Wilson \cite{WMultiset})
\label{valley-frobenius}
For a partition $\lambda \vdash n$, we have that the coefficient of $s_{\lambda}$
in $\Val_{n,k-1}(x;q,0) = \Val_{n,k-1}(x;0,q)$ equals
\begin{equation}
\label{schur-expansion}
\sum_{T \in \mathrm{SYT}(\lambda)} q^{\maj(T) + {n-k \choose 2} - (n-k) \cdot \mathrm{des}(T)} 
{ \mathrm{des}(T) \brack n-k}_q.
\end{equation}
In particular, we have that $\Val_{n,k-1}(x; 0,q)$ is Schur positive.
\end{corollary}

\begin{proof}
For $\Val_{n,k-1}(x;q,0)$, this follows from letting $m = 0$ and taking the coefficient 
of $u^{n-k}$ in \cite[Theorem 5.0.1]{WMultiset}.  Now apply 
Theorem~\ref{valley-function-equality}.
\end{proof}

\begin{example}
The calculation
\begin{center}
$\begin{array}{ccc}
\minimaj(1|23) = 0 & \minimaj(2|13) = 1 & \minimaj(3|12) = 1 \\
\minimaj(12|3) = 0 & \minimaj(13|2) = 1 & \minimaj(23|1) = 2
\end{array}$
\end{center}
shows that the coefficient of $x_1 x_2 x_3$ in 
$\Val_{3,1}(x;0,q)$ is $q^2 + 3q + 2$.
The calculation
\begin{center}
$\begin{array}{cc}
\minimaj(1|12) = 0 & \minimaj(12|1) = 1
\end{array}$
\end{center}
shows that the coefficient of $x_1^2 x_1$ in 
$\Val_{3,1}(x;0,q)$ is $q + 1$.
Since there are no ordered multiset partitions of weight $(3)$ with
$2$ blocks, the coefficient of $x_1^3$ in $\Val_{3,1}(x;0,q)$ is $0$.  
We therefore have the monomial basis expansion of the symmetric function $\Val_{3,1}(x;0,q)$ given by
\begin{equation*}
\Val_{3,1}(x;0,q) = (q^2 + 3q + 2) m_{(1,1,1)} + (q+1) m_{(2,1)},
\end{equation*}
which leads to the Schur basis expansion
\begin{equation*}
\Val_{3,1}(x;0,q) = (q^2 + q) s_{(1,1,1)} + (q + 1) s_{(2,1)}.
\end{equation*}
To see that this agrees with Corollary~\ref{valley-frobenius}, one may check that the four standard Young
tableaux with three boxes
{\em \begin{center}
\begin{Young}
1 \cr
2 \cr 
3
\end{Young}  \hspace{0.1in}
\begin{Young}
1 & 2 \cr
3
\end{Young}  \hspace{0.1in}
\begin{Young}
1 & 3 \cr
2
\end{Young} \hspace{0.1in}
\begin{Young}
1 & 2 & 3
\end{Young}
\end{center}}
\noindent
contribute (from left to right)
$(q^2 + q) s_{(1,1,1)}, q s_{(2,1)}, s_{(2,1)},$ and $0 s_{(3)}$
to the Schur expansion of $\Val_{3,1}(x;0,q)$ predicted by
(\ref{schur-expansion})
\end{example}

In contrast to the state of affairs regarding $\Val_{n,k}(x;q,t)$, it is known that the 
other 
quasisymmetric function $\Rise_{n,k}(x;q,t)$ 
appearing in the Delta Conjecture
is both symmetric in the $x$-variables and 
Schur positive \cite{HRW}.  However, the proof of Schur positivity relies
on the positivity of LLT polynomials -- a positivity result which lacks a combinatorial proof.

Wilson \cite{WMultiset}  proved combinatorially that we have the identities
\begin{equation}
\Rise_{n,k}(x;q,0) = \Rise_{n,k}(x;0,q) = \Val_{n,k}(x;q,0).
\end{equation}
It follows that the coefficient of $s_{\lambda}$ in the Schur basis expansion of 
$\Rise_{n,k-1}(x,q,0)$ or $\Rise_{n,k-1}(x;0,q)$ is also given by
(\ref{schur-expansion}).

\section{Colored set partitions and wreath products}
\label{Colored}

\subsection{Colored Permutations and Statistics}
A natural generalization of permutations can be obtained by coloring their entries.  
Let $n, r > 0$ and consider the alphabet
\begin{equation}
\AAA_{n,r} := \{i^{j} \,:\, 1 \leq i \leq n, 0 \leq j \leq r-1\}.  
\end{equation}
We think of $i^{j}$ as the letter $i$ augmented 
with the `color' $j$.
The {\em wreath product} $C_r \wr S_n$ consists of all length $n$ words 
$\pi_1 \dots \pi_n$ in $\AAA_{n,r}$ which contain just one copy of the letter $i$
(of any color) for all $1 \leq i \leq n$.  For example, we have
$\pi = 3^0 4^2 5^0 1^2 2^1 \in C_3 \wr S_5$.  We can also think of elements in $C_r \wr S_n$
as pairs $(\pi, \epsilon)$, where $\pi \in S_n$ is the uncolored permutation and 
$\epsilon = (\epsilon_1, \dots, \epsilon_n) \in \{0, 1, \dots, r-1\}^n$ is the {\em color sequence}.  In this language,
the colored permutation $3^0 4^2 5^0 1^2 2^1$ would be written as 
$(34512, 02021)$.

As a group, the wreath product $C_r \wr S_n$ is isomorphic to the complex reflection
group $G(r,1,n)$ consisting of $n \times n$ monomial complex matrices whose nonzero entries
are $r^{th}$ roots of unity.  We will have no occasion to use the group structure on 
$C_r \wr S_n$.

One extension of major index to colored permutations which has appeared in the literature is
as follows.  Introduce the following total order $\prec$ on the set $\AAA_{n,r}$:
\begin{equation}
1^{r-1} \prec \cdots \prec n^{r-1} \prec \cdots \prec 1^{1} \prec \cdots \prec n^{1} \prec
1^{0} \prec \cdots \prec n^{0}.
\end{equation}
Let $\pi = \pi_1 \dots \pi_n \in C_r \wr S_n$ have color sequence 
$\epsilon = (\epsilon_1, \dots, \epsilon_n)$.  The {\em major index} of $\pi$ is
\begin{equation}
\maj(\pi) := \left(r \sum_{\pi_{i+1} \prec \pi_i} i \right) + \sum_{i = 1}^n \epsilon_i.
\end{equation}
This statistic is called the  `flag major index' in \cite{AR, HLR}.  Inside $C_3 \wr S_5$ we have
\begin{equation*}
\maj(3^0 4^2 5^0 1^2 2^1) = 3(1 + 3) + (0 + 2 + 0 + 2 + 1) = 17.
\end{equation*}

A {\em $C_r \wr S_n$-ordered set partition} is a sequence 
$\sigma = (B_1 | \dots | B_k)$ of nonempty subsets of $\AAA_{n,r}$ such that, for all $1 \leq i \leq n$,
precisely one copy of $i$ appears (with any color) in $\sigma$.
As before, we call the sets $B_1, \dots, B_k$ the {\em blocks} of $\sigma$, we say that 
$\sigma$ has {\em $k$ blocks}, and we call the composition
$\alpha = (\alpha_1, \dots, \alpha_k) \models n$ the {\em shape} of $\sigma$.
For example, we have that 
\begin{equation*}
\sigma = (2^0 3^2 | 4^0 | 5^0 1^1)
\end{equation*}
 is a $C_3 \wr S_5$-ordered set partition  with $3$ blocks
of shape $(2,1,2) \models 5$.
We have the collections of $C_r \wr S_n$-ordered set partitions:
\begin{align}
\OP^r_n &:= \{ \text{all $C_r \wr S_n$-ordered set partitions} \}, \\
\OP^r_{n,k} &:= \{ \text{all $C_r \wr S_n$-ordered set partitions with $k$ blocks} \}, \\
\OP^r_{n,\alpha} &:= \{ \text{all $C_r \wr S_n$-ordered set partitions of shape $\alpha$} \}.
\end{align}

Let $\sigma = (B_1 | \dots | B_k)$ be a $C_r \wr S_n$-ordered set partition.  We define $C \wr S(\sigma)$ to be the set of elements of 
$C_r \wr S_n$ obtained by permuting the letters within the blocks of $\sigma$.  
For example,
\begin{equation*}
C \wr S(2^0 3^2 | 4^0 | 5^0 1^1) = \{ 2^0 3^2  4^0  5^0 1^1, 3^2 2^0  4^0  5^0 1^1, 2^0 3^2  4^0  1^1 5^0, 3^2 2^0  4^0   1^1 5^0 \}.
\end{equation*}
The set $C \wr S (\sigma)$ may be regarded as a parabolic coset in the complex reflection group
$C_r \wr S_n$.
The $\minimaj$ statistic on $\OP^r_n$ is defined by
\begin{equation}
\minimaj(\sigma) := \min\{\maj(\pi) \,:\, \pi \in C \wr S(\sigma) \}.
\end{equation}

In order to prove an analog of Theorem~\ref{same-distribution-permutations} 
for colored ordered set partitions, we need to define a suitable 
generalization of $\inv$ on $\OP^r_n$ for $r > 1$.  Let $d: \AAA_{n,r} \rightarrow [n]$ be the decolorizing map given by
$d(i^j) := i$ and extend $d$ to a map $d: \OP^r_n \rightarrow \OP_n$ in the obvious way.  Also define a statistic 
$\epsilon$ on $\OP^r_n$ by letting $\epsilon(\sigma)$ be the sum of the colors of the letters in $\sigma$.  We define
a statistic $\inv$ on $\OP^r_n$ by
\begin{equation}
\inv(\sigma) := r \cdot \inv(d(\sigma)) + \epsilon(\sigma),
\end{equation}
where $\inv(d(\sigma))$ is computed as in Section~\ref{Permutations}.
For example, if $\sigma = (1^1 2^2 4^0 | 7^2 | 8^1 9^2 | 3^1 5^1 6^2) \in \OP^3_9$, we have
\begin{equation*}
d(\sigma) = (124|7|89|356),
\end{equation*}
so that
\begin{equation*}
\inv(\sigma) = 3 (1+1+2)+(1+2+0+2+1+2+1+1+2) = 24.
\end{equation*}

We have the usual host of generating functions:

\begin{equation}
\begin{array}{cc}
\begin{cases}
I_{n}^r(q) := \sum_{\sigma \in \OP^r_n} q^{\inv(\sigma)}, \\
I_{n, k}^r(q) := \sum_{\sigma \in \OP^r_{\beta, k}} q^{\inv(\sigma)}, \\
I_{n, \alpha}^r(q) := \sum_{\sigma \in \OP^r_{n, \alpha}} q^{\inv(\sigma)},
\end{cases} & 
\begin{cases}
M^r_{n}(q) := \sum_{\sigma \in \OP^r_{n}} q^{\minimaj(\sigma)}, \\
M^r_{n, k}(q) := \sum_{\sigma \in \OP^r_{n, k}} q^{\minimaj(\sigma)}, \\
M^r_{n, \alpha}(q) := \sum_{\sigma \in \OP^r_{n, \alpha}} q^{\minimaj(\sigma)}.
\end{cases}
\end{array}
\end{equation}

\subsection{$M^r_{n,\alpha}(q) = I^r_{n,\alpha}(q) = ([r]_q)^n F_{n,\alpha}(q^r)$}

Recall the polynomial $F_{n,\alpha}(q)$ from Section~\ref{Permutations}.
We will prove a natural generalization of the equidistribution result of 
Theorem~\ref{same-distribution-permutations} to colored permutations.
Our approach here will be to reduce the colored result to the uncolored case of 
Theorem~\ref{same-distribution-permutations}.  In the next subsection we will see that trying to
mimic the proof of Theorem~\ref{same-distribution-permutations} itself leads to a 
strange and difficult polynomial identity.

\begin{theorem}
\label{same-distribution-colored}  Let $r, n \geq 1$ and let $\alpha \models n$ be a composition.  We have
\begin{equation}
I^r_{n, \alpha}(q) = M^r_{n, \alpha}(q) = ([r]_q)^n F_{n,\alpha}(q^r).
\end{equation}
\end{theorem}

\begin{proof}
We start by proving that $I^r_{n, \alpha}(q) = ([r]_q)^n F_{\alpha}(q^r)$.  Consider a fixed uncolored 
ordered set partition
$\sigma_0 \in \OP_{n, \alpha}$.  We have that 
$\sum_{\sigma} q^{\inv(\sigma)} = ([r]_q)^n q^{r ( \inv(\sigma_0))}$,
where the sum is over all colored ordered set partitions $\sigma \in \OP^r_{n, \alpha}$ satisfying
$d(\sigma) = \sigma_0$.  The identity $I^r_{n, \alpha}(q) = ([r]_q)^n F_{\alpha}(q^r)$ follows from summing over all
$\sigma_0 \in \OP_{n, \alpha}$ and applying 
Theorem~\ref{same-distribution-permutations}.

Next, we prove that $M^r_{n, \alpha}(q) = ([r]_q)^n F_{\alpha}(q^r)$.  To see this, we define a standardization map
$s: \OP^r_{n, \alpha} \rightarrow \OP_{n, \alpha}$ as follows.  Let $\sigma \in \OP^r_{n, \alpha}$ be a colored
ordered set partition and consider its elements with respect to the order $\prec$ on the alphabet
$\mathcal{A}_{n, r}$.  Let $s(\sigma)$ be 
the unique element of $\OP_{n, \alpha}$ which is order isomorphic to $\sigma$ under $\prec$.
For example, we have
\begin{equation*}
s(1^1 2^2 4^0 | 7^2 | 8^1 9^2 | 3^1 5^1 6^2) = (519|3|84|672) = (159|3|48|267).
\end{equation*}

Let $\sigma_0 \in \OP_{n, \alpha}$ be an uncolored ordered set partition and let 
$\gamma = (\gamma_0, \gamma_1, \dots, \gamma_{r-1}) \models_0 n$ be a weak composition of $n$ of length
$r$.  If $\sigma \in \OP^r_{n, \alpha}$ satisfies $s(\sigma) = \sigma_0$ and has $\gamma_i$
letters of color $i$ for all $i$, we have that 
\begin{equation}
\minimaj(\sigma) = r (\minimaj(\sigma_0)) + \sum_{i = 0}^{r-1} i \gamma_i.
\end{equation}

For a fixed $\sigma_0 \in \OP_{n, \alpha}$ and a fixed color multiplicity sequence
$\gamma = (\gamma_0, \gamma_1, \dots, \gamma_{r-1}) \models_0 n$, how many 
ordered multiset partitions lie in the set
\begin{equation*}
\{ \sigma \in \OP^r_{n,\alpha} \,:\, 
\text{$s(\sigma) = \sigma_0$ and $\sigma$ has $\gamma_i$ letters of color $i$ for $0 \leq i \leq r-1$} \}?
\end{equation*}
For any color $i$ and any block $B_j$, the number of entries 
in $B_j$ of color $i$ must equal the number of entries in the $j^{th}$ block of $\sigma_0$ which lie in the
interval $[\gamma_1 + \cdots + \gamma_{i-1} + 1, \gamma_1 + \cdots + \gamma_i]$.
To determine $\sigma$, we need only assign $\gamma_i$ entries in $[n]$ to have color $i$ in $\sigma$,
for all $0 \leq i \leq r-1$.  The number of ways to do this is the multinomial coefficient
${n \choose \gamma_0, \dots, \gamma_{r-1}}$.  We conclude that
\begin{equation*}
M^r_{n, \alpha}(q) = 
\sum_{(\gamma_0, \dots, \gamma_{r-1}) \models_0 n} {n \choose \gamma_0, \dots, \gamma_{r-1}}
q^{\sum_{i = 0}^{r-1} i \gamma_i} M_{n, \alpha}(q^r) \\
= ([r]_q)^n F_{n, \alpha}(q^r),
\end{equation*}
where the second equality used the Multinomial Theorem together with 
Theorem~\ref{same-distribution-permutations}.
\end{proof}

In light of Theorem~\ref{same-distribution-word}, it is natural to expect an equidistribution result
involving $\inv$ and $\minimaj$ on appropriately defined `colored ordered multiset partitions'
with a fixed number of blocks and a fixed weight.  
The author has been unable to produce such a result.
Equidistribution fails to hold in small examples when one considers 
sequences
of nonempty subsets of $\mathcal{A}_{n,r}$, e.g.
\begin{equation*}
(1^0 1^3 2^0 | 1^3 3^2 | 2^0 2^1 3^1).
\end{equation*}
Equidistribution also fails in small examples when one considers sequences of 
nonempty subsets of $\mathcal{A}_{n,r}$ where no set contains two copies of the same 
letter (of any color), e.g.
\begin{equation*}
(1^0 2^3 3^2 | 1^2 2^0 | 1^2 2^0 3^1).
\end{equation*}

\subsection{A Strange Polynomial Identity}
The reader may wonder whether the methods of Section~\ref{Permutations} can be used
to prove Theorem~\ref{same-distribution-colored} without relying
on Theorem~\ref{same-distribution-permutations}.  Attempting to do so leads to a
strange and complicated polynomial identity
whose truth is implied by Theorem~\ref{same-distribution-colored}, but which 
the author does not know how to prove
directly.

Lemma~\ref{cycle-major-permutations} generalizes naturally to the context of colored permutations.
We have an action of $\ZZ_{nr}$ on the alphabet $\AAA_{n,r}$ given by the cycle 
\begin{equation*}
c := (n^{0}, (n-1)^{0}, \dots, 1^{0},   n^1, (n-1)^1, \dots, 1^1,  \dots, n^{r-1}, (n-1)^{r-1}, \dots, 1^{r-1}) \in C_r \wr S_n.
\end{equation*}
Viewing $C_r \wr S_n$ as the complex reflection group $G(r,1,n)$, we can think of $c$ as the regular
element
\begin{equation*}
c = \begin{pmatrix}
 & 1 & & & \\
 & & 1 & & \\
 & & & \ddots & \\
 & & & & 1 \\
e^{\frac{2 \pi i}{r}} & & & &
\end{pmatrix}.
\end{equation*}
This action extends to an action $\pi \mapsto c.\pi$ on $C_r \wr S_n$ by setting
$c.\pi = (c.\pi_1) \dots (c.\pi_n)$ for $\pi = \pi_1 \dots \pi_n$.  The analog of Lemma~\ref{cycle-major-permutations}
continues to hold.

\begin{lemma}
\label{cycle-major-colored-permutations}
Let $\pi = \pi_1 \dots \pi_n \in C_r \wr S_n$ and assume that $\pi_n \neq 1^{r-1}$.  Then
$\maj(c.\pi) = \maj(\pi) + 1$.
\end{lemma}

\begin{proof}
If $\pi$ does not contain $1^{r-1}$, the application of $c$ preserves the descent set of $\pi$, sends $1^j$ to $n^{j+1}$, and preserves
all other signs.  If $\pi$ contains $1^{r-1}$, the application of $c$ moves the descent of $\pi$ ending in $1^{r-1}$ one unit to the right
while decreasing the sum of signs by $r-1$, resulting in a net increase of $\maj$ by $1$.
\end{proof}

The map $\pi: \OP_n \rightarrow S_n$ of Section~\ref{Permutations} extends easily to define a map
$\pi: \OP_n^r \rightarrow C_r \wr S_n$ for $r > 1$.  Given $(B_1 | \dots | B_k) \in \OP^r_{n,k}$, the $k$-segmented word
$\pi(B_1 | \dots | B_k) = \pi[1] \cdot \ldots \cdot \pi[k] \in C \wr S(B_1 | \dots | B_k)$ is defined using 
Definition~\ref{segment-pi-map}
with respect to the  order $\prec$ on the alphabet $\AAA_{n,r}$.  
For example, if $\sigma = (1^1 2^2 4^0 | 7^2 | 8^1 9^2 | 3^1 5^1 6^2) \in \OP^3_9$, we have
\begin{equation*}
\pi(\sigma) = \pi[1] \cdot \pi[2] \cdot \pi[3] \cdot \pi[4]  = 1^1 4^0 2^2 \cdot 7^2 \cdot 9^2 8^1 \cdot 6^2 3^1 5^1.
\end{equation*}

The analog of Lemma~\ref{minimaj-alternative} is still true in this setting.

\begin{lemma}
\label{minimaj-alternative-colored}
Let $\sigma$ be a $C_r \wr S_n$-ordered set partition.  The statistic $\maj$ achieves a unique minimum on the set
$C \wr S(\sigma)$, and this value is achieved at $\pi(\sigma)$.
\end{lemma}

The analog of Lemma~\ref{compress} also 
holds for colored set partitions.  Let $(C_1 | \dots | C_k)$ be a sequence of nonempty
subsets of $\AAA_{n,r}$ with the property that 
at most one of $i^0, i^1, \dots, i^{r-1}$ appears among the sets $C_1, \dots, C_k$,
for all $1 \leq i \leq n$. Let $\overline{(C_1 | \dots | C_k)}$ be the unique $C_r \wr S_{n'}$-ordered set partition
(for some $n' \leq n$) whose associated uncolored ordered set partition is order isomorphic to 
the uncolored sequence of sets $(C_1 | \dots | C_k)$, while preserving colors.  For example, we have
\begin{equation*}
\overline{(2^3 7^0 9^1 | 5^1 | 3^1 8^0)} = (1^3 4^0 6^1 | 3^1 | 2^1 5^0 ).
\end{equation*}

\begin{lemma}
\label{compress-colored}
Let $(B_1 | \dots | B_{k-1} | B_k)$ be a $C_r \wr S_n$-ordered set partition.  Write the elements of $B_k$  
as $B_k = \{ \beta_1 \prec \dots \prec \beta_m \}$ and write $\epsilon(\beta_i)$ for the color of 
$\beta_i$ for $1 \leq i \leq m$.  We have that
\begin{equation*}
\minimaj(B_1 | \dots | B_{k-1} | B_k) = \minimaj \overline{(B_1 | \dots | B_{k-1} | \beta_1)} + \epsilon(\beta_2) + \cdots + \epsilon(\beta_m).
\end{equation*}
\end{lemma}

\begin{proof}
Similar to the proof of Lemma~\ref{compress}, except one must take into account the contribution of colors to the $\minimaj$ statistic.
\end{proof}

Lemma~\ref{cycle-commute-permutations} generalizes in the obvious way.

\begin{lemma}
\label{cycle-minimizers-colored}
Let $\sigma \in \OP^r_{n,\alpha}$ be a $C_r \wr S_n$-ordered set partition of type $\alpha = (\alpha_1, \dots, \alpha_k) \models n$ with
$\alpha_k = 1$.  We have
\begin{equation}
\pi(c.\sigma) = c.\pi(\sigma).
\end{equation}
\end{lemma}

\begin{proof}
Similar to the proof of Lemma~\ref{cycle-commute-permutations}.
\end{proof}

Let $\alpha = (\alpha_1, \dots, \alpha_k) \models n$ and let
$\overline{\alpha} := (\alpha_1, \dots, \alpha_{k-1}) \models n - \alpha_k$.
By Theorem~\ref{same-distribution-colored}, we have the recursion
\begin{equation*}
M^r_{\alpha,n}(q) = ([r]_q)^{\alpha_k} \times 
\left[ {\alpha_k - 1 \choose \alpha_k - 1} + {\alpha_k \choose \alpha_k - 1}q^r + \cdots + {n-1 \choose \alpha_k - 1} q^{r(n-\alpha_k)} \right] 
\times M^r_{\overline{\alpha},n}(q).
\end{equation*}

On the other hand, we can derive a recursion for $M^r_{\alpha, n}(q)$ combinatorially.
Fix a letter $i^j \in \AAA_{n,r}$ for $1 \leq i \leq n$ and $0 \leq j \leq r-1$.
Let $\sigma = (B_1 | \cdots | B_{k-1} | B_k) \in \OP^r_{n,\alpha}$ be a typical
$C_r \wr S_n$-ordered set partition such that the minimum element $\min(B_k)$ (with respect to $\prec$) is $i^j$.  
The choices involved in the completion of the block $B_k$ are as follows:
\begin{enumerate}
\item  Choose a subset $S \subseteq \{i+1, i+2, \dots, n\}$ of size $t$ (for some $0 \leq t \leq \alpha_k - 1$).
\item  Choose a subset $T \subseteq \{1, 2, \dots, i-1\}$ of size $\alpha_k - t - 1$.
\item  Color the elements of $S$ with $\{0, 1, 2, \dots, j\}$ arbitrarily.
\item  Color the elements of $T$ with $\{0, 1, 2, \dots, j-1\}$ arbitrarily.
\end{enumerate}
For a fixed choice of $S$ and $T$ (and hence $t$) as above, the net contribution to $\minimaj(\sigma)$ coming from letter
colors in $B_k - \{\min(B_k)\}$ will be $[j]_q^{\alpha_k-t-1} [j+1]_q^t$.  On the other hand, 
Lemmas~\ref{cycle-major-colored-permutations}, \ref{compress-colored}, and \ref{cycle-minimizers-colored} imply that 
the choice of $\min(B_k) = i^j$ contributes an additional $j(n-\alpha_k+1) + (n-i-t)$ to $\minimaj(\sigma)$.  
We have ${n - i \choose t}$ choices for the set $S$ and ${i - 1 \choose \alpha_k - t - 1}$ choices for the set $T$.
Putting all of this together,
we get that 
\begin{multline}
M^r_{\alpha,n}(q) = M^r_{\overline{\alpha},n-\alpha_k}(q) \times \\
\sum_{i = 1}^n \sum_{j = 0}^{r-1} \sum_{t = 0}^{\alpha_k-1} {n-i \choose t} {i-1 \choose \alpha_k - t - 1}
q^{j(n-\alpha_k+1) + (n-i-t)} [j]_q^{\alpha_k-t-1} [j+1]_q^t.  
\end{multline}

The following polynomial identity follows from (and is equivalent to) Theorem~\ref{same-distribution-colored}.

\begin{proposition}
\label{polynomial-identity}
For all $1 \leq \alpha_k \leq n$ and all $r \geq 1$ we have
\begin{multline}
([r]_q)^{\alpha_k} \times 
\left[ {\alpha_k - 1 \choose \alpha_k - 1} + {\alpha_k \choose \alpha_k - 1}q^r + \cdots + {n-1 \choose \alpha_k - 1} q^{r(n-\alpha_k)} \right] = \\
\sum_{i = 1}^n \sum_{j = 0}^{r-1} \sum_{t = 0}^{\alpha_k-1} {n-i \choose t} {i-1 \choose \alpha_k - t - 1}
q^{j(n-\alpha_k+1) + (n-i-t)} [j]_q^{\alpha_k-t-1} [j+1]_q^t.  
\end{multline}
\end{proposition}

We challenge the reader to prove Proposition~\ref{polynomial-identity} directly!

\section{Closing Remarks}
\label{Closing}

In this paper we have applied the combinatorics of ordered set partitions to obtain evidence for
the Delta Conjecture of symmetric function theory.  Along the way to doing so, we discovered a new
family of polynomials
\begin{equation*}
F_{n,\alpha}(q) 
 = \prod_{i = 1}^k {\alpha_i - 1 \choose \alpha_i - 1} + {\alpha_i \choose \alpha_i - 1} q
 + {\alpha_i + 1 \choose \alpha_i - 1} q^2 + \cdots + 
 {\alpha_1 + \cdots + \alpha_i - 1 \choose \alpha_i - 1} q^{\alpha_1 + \cdots + \alpha_{i -1} }. 
\end{equation*}
indexed by compositions $\alpha = (\alpha_1, \dots, \alpha_k) \models n$.
We pose the problem of finding an algebraic interpretation of these polynomials.

Consider the action of the symmetric group $S_n$ on the polynomial ring
$\CC[x_1, \dots, x_n]$ given by variable permutation.  Let 
$I$ be the ideal generated by $S_n$-invariant polynomials 
$f(x_1, \dots, x_n) \in \CC[x_1, \dots, x_n]$ with vanishing constant term.  The 
{\em coinvariant algebra} is the graded $\CC$-vector space 
$R_n := \CC[x_1, \dots, x_n]/I$.  The space $R_n$ carries a graded action of $S_n$;
it can be shown that $R_n \cong \CC[S_n]$ as $S_n$-representations.

One natural $\CC$-basis of $R_n$ (the so-called {\em Artin basis}) consists of those monomials
$x_1^{a_1} x_2^{a_2} \cdots x_n^{a_n}$ whose exponent sequence 
$(a_1, a_2, \dots, a_n)$ satisfies $a_i < i$.  Said differently, the multiplication map gives
rise to an isomorphism of graded $\CC$-vector spaces
\begin{equation*}
\CC[x_1]_{\leq 0} \otimes \CC[x_2]_{\leq 1} \otimes \cdots \otimes \CC[x_n]_{\leq n-1} 
\xrightarrow{\sim} R_n.
\end{equation*}
Here $A_{\leq d}$ denotes the subspace of a graded $\CC$-algebra $A$ consisting
of elements of degree $\leq d$.
This shows that the Hilbert series $\Hilb(R_n; q)$
of $R_n$ is $[n]!_q = F_{n,(1^n)}(q)$.

For an arbitrary composition $\alpha \models n$, we have that $F_{n,\alpha}(q)$ is the generating function
for the degree statistic on monomials
$x_1^{a_1} \dots x_n^{a_n}$ whose exponent sequences $(a_1, \dots, a_n)$ satisfy
\begin{equation*}
a_{\alpha_1 + \cdots + \alpha_{i-1} + 1} + \cdots + a_{\alpha_1 + \cdots + \alpha_i} \leq \alpha_1 + \cdots + \alpha_i
\end{equation*}
for all $1 \leq i \leq k$.  Said differently, the polynomial $F_{n,\alpha}(q)$ is the Hilbert series of the 
graded $\CC$-vector space
\begin{equation*}
S_{n,\alpha} := \CC[x_1, \dots, x_{\alpha_1}]_{\leq 0} \otimes 
\CC[x_{\alpha_1 + 1}, \dots, x_{\alpha_1 + \alpha_2}]_{\leq \alpha_1} \otimes \cdots \otimes
\CC[x_{n-\alpha_k+1}, \dots, x_n]_{\leq n - \alpha_k}.
\end{equation*}

The Artin basis is one of several interesting bases for the coinvariant algebra $R_n$.  
The {\em descent monomial basis}
$\{d_{\pi} \,:\, \pi \in S_n\}$ introduced by Garsia and Stanton \cite{GarsiaStanton} has elements given by
\begin{equation*}
d_{\pi} = \prod_{i \in \mathrm{Des}(\pi)} x_{\pi_1} x_{\pi_2} \cdots x_{\pi_i}.
\end{equation*}
It is evident that the degree of the monomial $d_{\pi}$ is given by $\maj(\pi)$.  The monomials 
$\{ d_{\pi} \,:\, \pi \in S_n\}$ may be extended to ordered set partitions by setting
$d_{\sigma} := d_{\pi(\sigma)}$ for any ordered set partition $\sigma$.  It is immediate that the degree
of $d_{\sigma}$ is $\minimaj(\sigma)$.

A ubiquitous basis of $R_n$ which does {\em not} consist of monomials 
is the {\em Vandermonde basis} whose elements are all of the nonzero
polynomials which can be obtained by applying the partial derivative operators 
$\partial_{x_1}, \dots, \partial_{x_n}$ to the Vandermonde polynomial 
$\prod_{1 \leq i < j \leq n} (x_i - x_j)$.
Yet another important basis is 
the {\em Schubert basis} which
 has elements given by all nonzero polynomials which can be obtained
by applying the divided difference operators 
$\sigma_1, \dots, \sigma_{n-1}$ to the monomial $x_1^{n-1} \cdots x_{n-1}^1 x_n^0$.
The degrees of Vandermonde and Schubert basis elements are closely related to the 
$\maj$ and $\inv$ statistics on the symmetric group.

\begin{problem}
Find analogs of the Vandermonde and Schubert basis for an arbitrary composition
$\alpha \models n$.  Find an analog $I_{n,\alpha}$ of the ideal $I_n \subseteq \CC[x_1, \dots, x_n]$ 
generated by homogeneous symmetric polynomials of positive degree.  The multiplication map
should give a graded $\CC$-vector space isomorphism
$S_{n,\alpha} \rightarrow R_{n,\alpha} := \CC[x_1, \dots, x_n]/I_{n,\alpha}$.   The analogs of
the Artin, descent monomial, Vandermonde, and Schubert bases should give bases of $R_{n,\alpha}$.
\end{problem}

The coinvariant algebra $R_n$ carries a graded action of $S_n$ whose graded Frobenius
character is the Hall-Littlewood polynomial $P_{(1^n)}(x;q)$.  It is well known that for 
$\lambda \vdash n$, the coefficient of $s_{\lambda}$ in the Schur expansion of 
$P_{(1^n)}(x;q)$ equals 
$\sum_{T \in \mathrm{SYT}(\lambda)} q^{\mathrm{maj}(T)}$.
This coincides with the expression appearing in Corollary~\ref{valley-frobenius}
when $k = n$.
We ask for a Delta Conjecture-style generalization of the coinvariant module.

\begin{problem}
For any $k \leq n$, find a nice graded $S_n$-module $R_{n,k}$ whose graded Frobenius character
is $\Val_{n,k}(x;q,0) = \Val_{n,k}(x;0,q)$ (after applying the involution $\omega$ which exchanges
$h_n$ for $e_n$ and complementing $q$-degree).
\end{problem}

In the course of proving the equality $\Val_{n,k}(x;q,0) = \Val_{n,k}(x;0,q)$ it was necessary to prove that
$\Val_{n,k}(x;0,q)$ is symmetric in the $x$ variables.  We accomplished this by introducing a family
$\{t_1, t_2, \dots \}$ of involutive switch maps on ordered multiset partitions which preserve
minimaj but permute weight compositions.

The {\em Bender-Knuth operators} $\{\tau_1, \tau_2, \dots \}$ are a famous collection of involutions on
semistandard Young tableaux.  Given a SSYT $T$, the operator $\tau_i$ acts on $T$ by freezing every 
instance of $i$ immediately above $(i+1)$, and then interchanging the number of $i$'s and
$(i+1)$'s in every maximal unfrozen horizontal run consisting entirely of $i$'s and $(i+1)$'s.
For example, we have (writing frozen entries in bold)

\begin{center}
\begin{Young}
, \cr
, $\tau_2:$   \cr
, \cr
\end{Young} \hspace{0.05in}
\begin{Young}
1 & 1 & {\bf 2} & 2 & 3 & 3 & 3 \cr
{\bf 2} & 2 & {\bf 3} & 4 & 4 \cr
{\bf 3} & 4 & 5
\end{Young}
\begin{Young}
, \cr
, $\mapsto$  \cr
, \cr
\end{Young}
\begin{Young}
1 & 1 & {\bf 2} & 2 & 2 & 2 & 3 \cr
{\bf 2} & 3 & {\bf 3} & 4 & 4 \cr
{\bf 3} & 4 & 5
\end{Young}.
\end{center}

Since the operator $\tau_i$ interchanges the number of $i$'s and $(i+1)$'s in any SSYT, the Bender-Knuth
operators define a family of involutions $\{t'_1, t'_2, \dots \}$ on words which preserve descent set 
(and, in particular, major index) while permuting weight.  In particular, let $w = w_1 \dots w_n$ be a word. 
We may apply Robinson-Schensted-Knuth 
insertion to $w$ to get a pair $w \mapsto (P, T)$, where $P$ and $T$ are tableaux
of the same shape, the tableau $P$ is standard, the tableau $T$ is semistandard, and $T$ has content
equal to the weight of $w$. For any $i$, let $t'_i(w)$ be the unique word whose RSK image is
$(P, \tau_i(T))$.  
The desired involutive and weight action properties of the $t'_i$ are inherited from the corresponding
properties of the $\tau_i$.

It turns out that the word operators $t_i$ and $t'_i$ do not coincide.  In fact, the switch maps $t_i$ do
not even preserve the shape of a word under RSK.  Given the importance of the RSK algorithm,
it would be interesting to extend the definition of the $t'_i$ operators to ordered multiset partitions
in such a way that 
\begin{itemize}
\item  the $t'_i$ act as involutions on ordered multiset partitions,
\item  the $t'_i$ preserve $\minimaj$, and
\item  the $t'_i$ extend appropriately defined Bender-Knuth operators under the image of 
an appropriate ordered multiset partition version of RSK.
\end{itemize}

\section{Acknowledgements}
\label{Acknowledgements}

The author is very grateful to Nantel Bergeron,
Jim Haglund, Ira Gessel, Jeff Remmel, Anne Schilling, and Andy Wilson for many 
inspiring  discussions and a great deal of constructive input on this paper.
The author was partially supported by NSF Grants
DMS-1068861 and DMS-1500838.

\end{document}